\documentclass[12pt,epsfig,amsfonts]{amsart} 
\usepackage{amsmath,amsthm,amssymb,amscd,epsfig,color}
\usepackage{graphicx}
\usepackage{ulem}
\usepackage{mathrsfs}

\newtheorem*{theorema}{Main Theorem}

\newtheorem{prop}{Proposition}[section]
\newtheorem{lemma}[prop]{Lemma}

\theoremstyle{remark}
\newtheorem{remark}[prop]{Remark}

\numberwithin{equation}{section}

\begin{document}

\author{Johannes Jaerisch and Hiroki Takahasi}

\address{Graduate School of Mathematics, Nagoya University,
Furocho, Chikusaku, Nagoya, 464-8602, JAPAN} 
\email{jaerisch@math.nagoya-u.ac.jp}

\address{Keio Institute of Pure and Applied Sciences (KiPAS), Department of Mathematics,
Keio University, Yokohama,
223-8522, JAPAN} 
\email{hiroki@math.keio.ac.jp}


\title[multifractal analysis of Homological growth rates]{
Multifractal analysis of
homological\\ growth rates
for hyperbolic surfaces} 
 \date{\today}
 \maketitle
 \begin{abstract} 
 We perform a multifractal analysis of homological growth rates of oriented geodesics on hyperbolic surfaces.  
 Our main result provides a formula for the Hausdorff dimension of level sets of prescribed growth rates in terms of a generalized Poincar\'e exponent of the Fuchsian group. We employ symbolic dynamics developed by Bowen and Series, ergodic theory and thermodynamic formalism to prove the analyticity of the dimension spectrum.
 \end{abstract}
 
 \section{Introduction}



A Fuchsian group is a discrete  group
of orientation-preserving isometries acting in the Poincar\'e disc model $(\mathbb D,d)$ of hyperbolic space. 
Fuchsian groups play an important role in the uniformization of hyperbolic surfaces  and geometric group theory.
 For the background  on Fuchsian groups we refer the reader to \cite{Bea83}.

 Throughout this paper, $G$ denotes a  finitely generated non-elementary Fuchsian group. Having fixed a finite set of generators of $G$, for $g\in G$ we denote by $|g|$ the minimal number of  generators needed to represent $g$, called the word length of $g$.
 It follows from the triangle inequality that there exists  $\alpha_+ >0$ such that $d(0,g0)\le \alpha_+ |g|$ for all $g\in G$. If  $\mathbb{D}/G$ has no cusps,  the Milnor-Swarc Lemma implies the existence of  $\alpha_->0$ such that $\alpha_- |g| \le d(0,g0)$. If $\mathbb{D}/G$ has cusps,  there exists $C>0$ such that $2\log|g|-C \le d(0,g0)$ by \cite{floyd80}. 
The complexity of the  
 action of $G$ is reflected in the fact that
 the  growth rate of $d(0,g0)/|g|$, as $|g|\rightarrow \infty$,  takes on uncountably many values, and rates of convergence are not uniform. In this paper 
 we perform a multifractal analysis of this growth rate  along oriented geodesics, which are  circular arcs orthogonal to the boundary  $\mathbb S^1$ of $\mathbb D$.

  \begin{figure}
\begin{center}
\includegraphics[height=5cm,width=10cm]{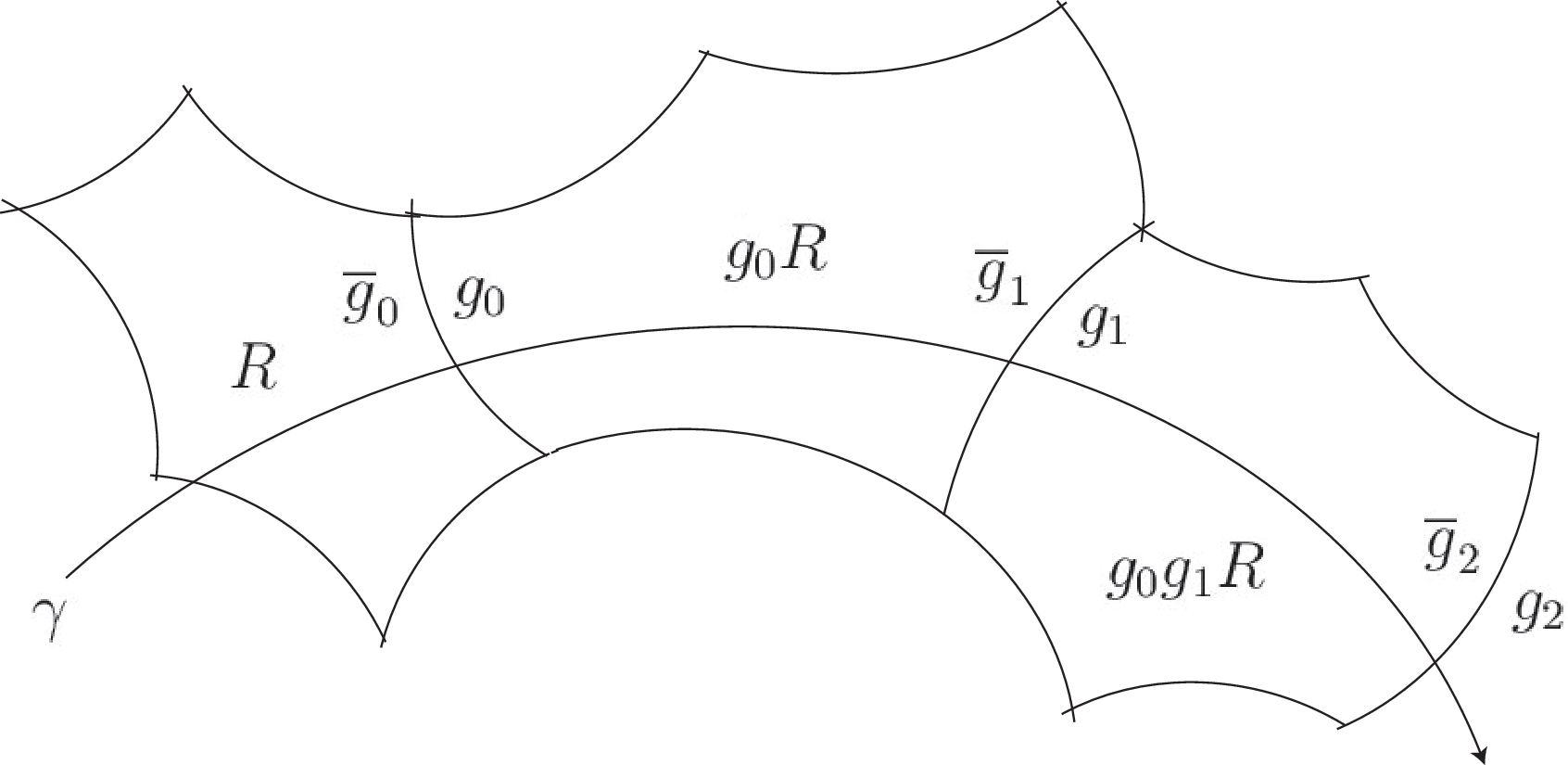}
\caption{An oriented geodesic $\gamma$ crossing  copies of the fundamental domain $R$.}
\label{fig:geo}
\end{center}
\end{figure}

  Let $R\subset \mathbb{D}$ be a convex,   locally finite fundamental domain for $G$ which contains $0$ in its interior \cite{Bea83}. The finite set of side-pairings of $R$ is denoted by $G_R$ and defines a symmetric set of generators of $G$. 
 We call $R$ {\it admissible} if $R$ has {\it even corners} \cite{BowSer79,Ser86} and satisfies a technical condition. We refer the reader to Section~\ref{sec:cutting} for the details.  
 Let $\mathscr{R}$ denote the set of oriented complete geodesics $\gamma$ joining two points  in $\mathbb S^1$ and intersecting the interior of $R$. If $\gamma \in \mathscr{R}$ cuts through the copies $R, g_0R, g_0g_1R, \ldots $ of $R$, with $g_i \in G_R$ and $i=0,1,\ldots\in\mathbb N$, then $g_{0}, g_{1}, g_{2}, \ldots$
  is called the {\it  cutting sequence} of $\gamma$ (see \textsc{Figure}~\ref{fig:geo}). By slightly perturbing geodesics cutting  through a vertex of $R$  we will define for each $\gamma\in \mathscr{R}$ a unique finite or infinite cutting sequence in Section \ref{sec:cutting}.  For $\gamma \in \mathscr{R}$ with  cutting sequence 
   $g_{0},g_{1},\ldots$
   of length at least $n\ge 1$,  we define
    \[t_{n}(\gamma)=d(0,g_0 g_1 \cdots g_{n-1}0),\] 
and we call  $t_n(\gamma)/n$ the homological growth rate  of $\gamma$ \cite{KesStr04}.  Since $R$ has even corners,    $g_0\cdots g_{n-1}$ has word length $n$ with respect to $G_R$
(see Proposition~\ref{cutting-is-shortest}).  We denote by $\Lambda=\Lambda(G)$ the limit set of $G$, and by $\Lambda_c=\Lambda_c(G)$ the conical limit set of $G$. We have $\Lambda_c\subset \Lambda$ and by a result of Beardon and Maskit \cite{BeaMask74}, $\Lambda \setminus \Lambda_c$ is  equal to the countable set of parabolic fixed points of elements of $G$. 
It turns out in Lemma~\ref{cut-converge} that  $\gamma \in \mathscr{R}$  has an infinite cutting sequence  if and only if its positive endpoint  $\gamma^+ $  belongs to $\Lambda_c$.
For $\alpha\geq0$ we define the {\it level set}
 \[\mathscr{H}(\alpha)=\left\{\xi\in\Lambda_c\colon\text{there exists $\gamma\in\mathscr{R}$
  \ such that $\gamma^+=\xi$ and } \lim_{n\to\infty}\frac{t_n(\gamma)}{n}=\alpha\right\}.\] 

    Since the  level sets are pairwise disjoint
   by  Remark~\ref{coincide-rem}, we have a multifractal decomposition of the conical limit set 
  \[\Lambda_c=\left(\bigcup_{\alpha\geq0} \mathscr{H}(\alpha)\right)\cup \mathscr{H}_{\rm ir}, \]
  where $\mathscr{H}_{\rm ir}$ denotes the set of $\xi\in\Lambda_c$ for which $t_n(\gamma)/n$ does not converge as $n\to\infty$ for any $\gamma\in\mathscr{R}$ whose positive endpoint is $\xi$.

    If $G$ is of the first kind, that is, $\Lambda=\mathbb{S}^1$, then
  there exists a 
  constant $\alpha_G\geq0$ such that
$\mathscr{H}(\alpha_G)$ has full Lebesgue measure in $\mathbb S^1$.
We refer the reader to Section~\ref{typical-h}
for more information on $\alpha_G$.  
For a description of the fine structure of $\Lambda$, it is necessary to
analyze other level sets which are negligible in terms of  Lebesgue measure. 
    Let $\dim_{\rm H}$ denote the Hausdorff dimension on $\mathbb S^1$, and for $\alpha\geq0$ let
    \[b(\alpha)=\dim_{\rm H} \mathscr H(\alpha) .\]
    Information on the complexity of the limit set is encoded in
 the function
 $\alpha\mapsto b(\alpha)$ called the
 {\it  spectrum of homological growth rates}, or simply the $\mathscr{H}$-spectrum. Note that the $\mathscr{H}$-spectrum depends on the choice of the fundamental domain $R$. 
 
The thermodynamic formalism gives an access to the description of the $\mathscr{H}$-spectrum. We define 
\[\delta_G=\dim_{\rm H}\Lambda.\]
It is well known \cite{Bea71,Pat76} \cite[Corollary~26]{Sul79} that  $\delta_G$ is equal to the {\it  Poincar\'e exponent of $G$} given by
\[\inf \left\{\beta\geq0  \colon \sum_{g\in G} \exp(-\beta d(0,g0))<+\infty\right\}.\]
 Imitating this style, following  \cite[Theorem~2.1.3]{MauUrb03}
we introduce a {\it  generalized Poincar\'e exponent} at an inverse temperature $\beta \in \mathbb{R}$ by
\[P(\beta)=\inf \left\{t \in \mathbb{R} \colon \sum_{g\in G} \exp(-\beta d(0,g0)-t|g|)<+\infty\right\}.\]
We call the function $\beta\in\mathbb R\mapsto P(\beta)$ the {\it  geometric pressure function of $G$ with respect to $R$}, or simply the {\it pressure function}. The convex conjugate of $P$ is for $\alpha \in \mathbb{R}$ given by
\[P^*(-\alpha)=\inf
\left\{\alpha\beta + P(\beta)\colon\beta\in\mathbb R\right\}.\]
We set \[\alpha_+=\sup_{\gamma\in\mathscr{R} ,\gamma^+ \in \Lambda_c}\limsup_{n\to\infty}\frac{t_n(\gamma)}{n}\quad
\text{and}\quad\alpha_-=\inf_{\gamma\in\mathscr{R} ,\gamma^+ \in \Lambda_c}\liminf_{n\to\infty}\frac{t_n(\gamma)}{n},\] 
and define the {\it freezing point} 
\[
\beta_+=\sup\{\beta\in\mathbb R\colon P(\beta)>-\alpha_-\beta\}.
\]

\begin{figure}
\begin{center}
\includegraphics[height=9cm,width=12cm]{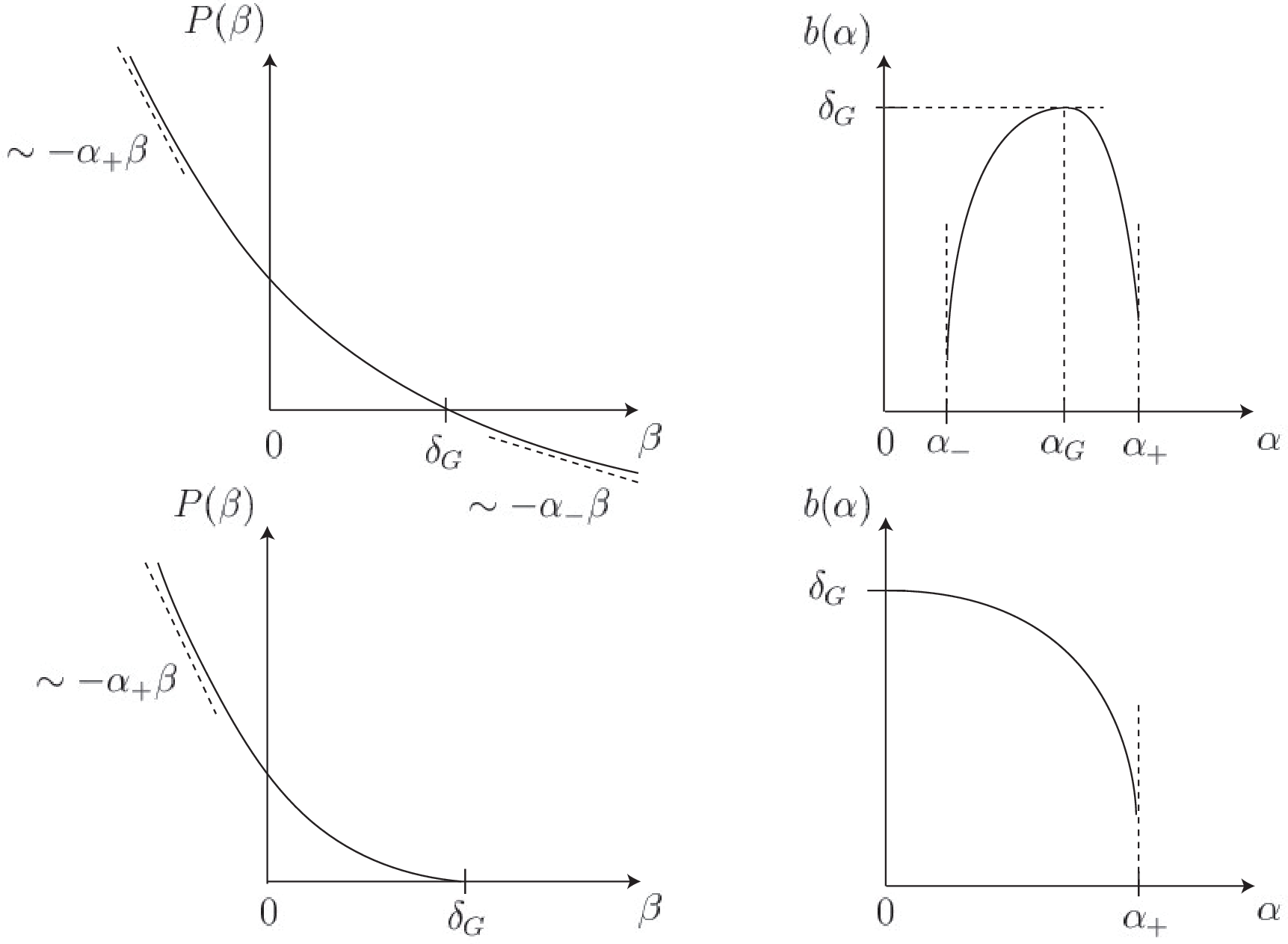}
\caption{The graphs of $\beta\in\mathbb R\mapsto P(\beta)$ and $\alpha\in[\alpha_-,\alpha_+]\mapsto b(\alpha)$: $G$ has no parabolic element (upper); $G$ has a parabolic element (lower).
We have $\delta_G=\min\{\beta\geq0\colon P(\beta)=0\}$. The constant $\alpha_G$ is the unique maximum point of the $\mathscr{H}$-spectrum, see \eqref{ag} for the definition.}
\label{fig:spectrum}\end{center}\end{figure}

\begin{theorema}
\label{bslyapunov}
Let $G$ be a finitely generated non-elementary Fuchsian group with an admissible fundamental domain $R$. 
\begin{itemize} 
\item[(a)] We  have $\alpha_-<\alpha_+$, and  the level set $\mathscr{H}(\alpha)$ is non-empty if and only if $\alpha\in[\alpha_-,\alpha_+]$.   The $\mathscr{H}$-spectrum is
continuous on $[\alpha_-,\alpha_+]$,  analytic on $(\alpha_-,\alpha_+)$, and for each $\alpha\in[\alpha_-,\alpha_+]\setminus\{0\}$ we have
\[b(\alpha)=\frac{P^*(-\alpha)}{\alpha}.\]  Moreover, the $\mathscr{H}$-spectrum attains its maximum $\delta_G$ at a unique $\alpha_G\in [\alpha_-,\alpha_+)$,  is strictly increasing on  $[\alpha_-,\alpha_G]$ and strictly decreasing on  $[\alpha_G,\alpha_+]$, and $\lim_{\alpha\nearrow\alpha_+}b'(\alpha)
=-\infty$. 
  If $G$ has no parabolic element, then  $\alpha_G>\alpha_->0$ and $\lim_{\alpha\searrow\alpha_-}b'(\alpha)=+\infty$.     If $G$ has a parabolic element, then  $\alpha_G=\alpha_-=0$.
    \item[(b)] 
 The pressure function $P$
is convex and continuously differentiable on $\mathbb R$, and  analytic and strictly convex on $(-\infty,\beta_+)$.
If $G$ has no parabolic element, then
 $\beta_+=+\infty$. If  $G$ has a parabolic element, then $\beta_+=\delta_G$ and $P$ vanishes on $[\delta_G,+\infty).$ 
     \end{itemize} 
     \end{theorema}

For finitely generated,  essentially free Kleinian groups in arbitrary dimension, Kesseb\"ohmer and Stratmann \cite{KesStr04} analyzed homological growth rates along  geodesic rays, 
and analyzed the $\mathscr{H}$-spectrum.
Our Main~Theorem  significantly extends \cite[Theorem~1.2]{KesStr04} to  a large class of Fuchsian groups which are not free groups. In particular, the Main~Theorem applies to Fuchsian groups uniformizing compact hyperbolic surfaces. 

  A key ingredient  in \cite{KesStr04}  is that
for essentially free Kleinian groups, cutting sequences of geodesic rays directly give a symbolic coding of the limit set by a Markov shift.
For Fuchsian groups, essentially free groups are free groups, and the  Koebe-Morse coding coincides with the Artin coding \cite{Ser86}.
For the Fuchsian groups we consider in this paper, the Koebe-Morse and Artin codings do not necessarily coincide  \cite{Ser86}, namely, cutting sequences do not have a direct link to the dynamics on the limit set. To overcome this difficulty, we utilize the  results for Fuchsian groups with even corners  \cite{BowSer79, Ser86}. 

The Main Theorem is a manifestation of the familiar thermodynamic  and  multifractal  picture for conformal expanding Markov maps possibly
with neutral fixed points (see e.g., \cite{GelRam09,Iom10,KesStr07,Nak00,Pes97,PesWei97,PesWei01,PreSla92,Sch99,Wei99})
in the context of Fuchsian groups.
Indeed, one main step in the proof of the Main~Theorem is to 
 clarify an elusive coincidence between the $\mathscr{H}$-spectrum and the Lyapunov spectrum of the Bowen-Series map \cite{BowSer79}.

 Let us compare 
 \cite[Theorem~1.2]{KesStr04} and the Main~Theorem in terms of phase transitions, 
 i.e., 
  the loss of analyticity of the pressure function in the case the group has a parabolic element.
For essentially free Kleinian groups,
two types of phase transitions were detected
in \cite[Theorem~1.2]{KesStr04}: the pressure is not differentiable at the freezing point, or the pressure is
continuously differentiable on $\mathbb R$ and not analytic at the freezing point.
For the Fuchsian groups considered in the Main~Theorem,
we have shown that only the second type of phase transition occurs.

In fact,
 the graphs of the $\mathscr{H}$-spectra in \textsc{Figure}~\ref{fig:spectrum} are only schematic. If $G$ has no parabolic element,
 we do not know whether the spectrum is concave
on $[\alpha_-,\alpha_+]$ or not (see \cite{IomKiw09}). Moreover,
if $G$ has a parabolic element and the pressure function is $C^2$, then 
$P''(\delta_G)=0$, which implies that the $\mathscr{H}$-spectrum has an inflection point (see Proposition~\ref{schematic-prop}).

\subsection*{Methods of proofs and structure of the paper}
Bowen and Series \cite{BowSer79} constructed a piecewise analytic Markov map $f\colon\varDelta\to\mathbb S^1$ which is orbit equivalent to the action of $G$ on the limit set, now called the {\it  Bowen-Series map}. 
To prove our main results,
 we use three different 
  symbolic codings (partitions) associated with the limit set $\Lambda$
  and the map $f$.
 
 In Section~2, following \cite{BowSer79,Ser86} we introduce the Bowen-Series map and a non-Markov partition well-adapted to the group structure, and develop various asymptotic results associated with them. 
 A main conclusion is that (I)
 the level sets of homological growth rates coincide with the level sets of the pointwise Lyapunov exponents of the map $f$ (Proposition~\ref{coincide-lem}).

 The Markov partition constructed in \cite{BowSer79} is an infinite partition if and only if $G$ has a parabolic element.
In Section~3, for groups having parabolic elements
we construct a finite Markov partition slightly modifying the construction in \cite{BowSer79}.
Combining this with the non-Markov partition introduced in Section~2,
 we show that (II)
 the generalized Poincar\'e exponent coincides with the geometric pressure
 (Proposition~\ref{p-equal}). 

 By virtue of the identities (I) and (II), the proof of the Main~Theorem boils down to 
 implementing the thermodynamic formalism and multifractal analysis for the map $f$.
 Series \cite[Theorem~5.1]{Ser81b} showed that 
 some power of $f$ is uniformly expanding  
  if $G$ has no parabolic element.
In this case, properties of the pressure function
and that of the Lyapunov spectrum of $f$ are well known \cite{Bow75,Pes97,PesWei97,PesWei01,Rue04,Wei99}.
If $G$ has a parabolic element, $f$ has a neutral periodic point
and these classical results do not apply.
To deal with this case, we take an inducing procedure
that is now familiar in the construction of equilibrium states (see e.g., \cite{PesSen08}).
In Section~4 we
construct a uniformly expanding induced Markov map $\tilde f$ equipped with an infinite Markov partition that allows us to represent
$\tilde f$ with a countable Markov shift.

Although the construction of the induced Markov map $\tilde f$ essentially follows 
Bowen and Series \cite{BowSer79},
one important difference from \cite{BowSer79} is that we dispense with the geometric hypothesis $(i)$ of property $(*)$ in \cite[p.406]{BowSer79} which states that {\it each side of the fundamental domain is contained in the isometric circle of the associated side-pairing.}
This implies that
$f$ is non-contracting, namely 
 \begin{equation}\label{NC}\inf_\varDelta |f'|\geq1.\end{equation}
 This kind of assumption is usually imposed 
in the thermodynamic formalism as well as the multifractal analysis 
of pointwise Lyapunov exponents of intermittent Markov maps, in order to facilitate arguments, 
see e.g., \cite{GelRam09,JaeTak20,JJOP10,JorRam11,MauUrb00,Nak00,PreSla92,Urb96,Yur02},
 and also \cite{Mor97}.
  We exploit the discrete group structure  and dispense with \eqref{NC} altogether.  
     If the Fuchsian group $G$ has no parabolic element, one can apply the Milnor-Swarc~Lemma to derive that some iterate of $f$ is uniformly expanding \cite{Ser81b}. If $G$
has parabolic elements, we use similar ideas to derive uniform expansion of the induced map $\tilde f$ (see Lemma~4.4 and Proposition~4.5). 
 

In Section~5 and Appendix A
we verify several conditions on  induced potentials associated with $\tilde f$, and apply results of Mauldin and Urba\'nski 
\cite{MauUrb03} (see also e.g., \cite{ADU93,BS03,Sar99}) to establish the existence and uniqueness of equilibrium states for the induced map $\tilde f$.
We then construct equilibrium states for the original map $f$, and use them to 
establish the analyticity of the pressure function.
Further, we combine results in the previous sections with the dimension formula for level sets of pointwise Lyapunov exponents in \cite{JaeTak20} to complete the proof of the Main~Theorem.


\subsection*{Notation}
Throughout we shall use the notation
 $a\ll b$ 
for two positive reals $a$, $b$
to indicate that 
 $a/b$ is bounded  from above  by a constant which depends only on  $G$ or $R$.
If $a\ll b$ and $b\ll a$, we write
$a\asymp b$.
For $g\in G$, the inverse of $g$ is denoted by $\bar g$, and the 
word length of $g$ with respect to $G_R$ is denoted by $|g|$.
Let ${\rm cl}(\cdot)$ and ${\rm int}(\cdot)$ denote the closure and interior operations in $\mathbb S^1$ respectively.
 Let $|\cdot|$ denote the Lebesgue measure on $\mathbb S^1$,
and let ${\rm diam}(\cdot)$ denote the Euclidean diameter on $\mathbb R^2$.

\section{The Bowen-Series map}

In Section~\ref{sec:cutting} we collect basic facts about  cutting sequences and fundamental domains with even corners.  In Section~\ref{BowSer} we introduce the Bowen-Series map $f$  together with an associated non-Markov symbolic coding
 called  $f$-expansion.
 In Section~\ref{character}, following Series \cite{Ser86} we characterize admissible words for this coding that will be used later. In Section~\ref{no-attr} we establish uniform decay of cylinders, and use it in Section~\ref{milddist} to prove a distortion estimate. 
 In Section~\ref{conv-cut} we describe 
 two orbits in the hyperbolic space, one from the cutting sequence
 of a geodesic and the other from the $f$-expansion of the positive endpoint of the same geodesic. 
 In Section~\ref{decayest} we relate the size of a cylinder with the corresponding homological growth rate, and use this estimate
 in Section~\ref{L=H} to show that
 the level sets of homological growth rates coincide with the level sets of the pointwise Lyapunov exponents of the map $f$.

\subsection{Cutting sequences for fundamental domains with even corners} \label{sec:cutting}
Let $R\subset \mathbb{D}$ be a  fundamental domain for $G$. By a fundamental domain we always mean a convex  and locally finite fundamental domain  which contains $0$ in its interior \cite{Bea83}.  The sides of $R$ are geodesics, or else arcs contained in $\mathbb S^1$.
The latter sides are called free sides. Note that $G$ is of the first kind 
if and only if $R$ has no free sides \cite[Theorem~12.2.12]{Rat94}. Since $G$ is finitely generated, $R$ has finitely many sides.
The sides of $R$ which are not free 
give rise to a finite set  of side-pairing transformations $G_R$. Recall that $G_R$ is a  symmetric set of generators of $G$. 

 The copies of $R$ adjacent to $R$ along the sides of $R$ are of the form $eR$,
$e\in G_R$. For every $g\in G$  and $e\in G_R$, we label the side common to $gR$ and $geR$ on the side of $geR$ by $e$, and  on the side of $gR$ by $\bar e$. 
By a side or vertex of 
$N=G\partial R$ 
we mean the $G$-image of a side or vertex of $R$. 
We say   $R$  has {\it  even corners} 
 if  $N$ is a union of complete geodesics  (\cite{Ser86}, see also \cite{BowSer79}).  
 We say $R$ is {\it admissible} if $R$ has even corners with  at least four sides and satisfies the following property: if  $R$ has precisely four sides with all vertices in $\mathbb{D}$ then at least three geodesics in $N$ meet at each vertex of $R$ \cite[Theorem~3.1]{Ser86}.
 The even corner assumption is not as restrictive as it appears. In fact, every surface which is uniformized by a finitely generated Fuchsian group
 has a fundamental domain with this property  (see  \cite[Section 3]{BowSer79} and \cite[p.609, l.9-10]{Ser86}).


Unless otherwise stated we assume all geodesics are complete.
  If $\gamma$ is an oriented geodesic which passes through a vertex $v$ of $N$ in $\mathbb D$, we make the convention that $\gamma$ is replaced by a curve deformed to the right around $v$.
We shall take as understood that all geodesics have been deformed, where necessary, in this way. 

   For  $\gamma \in \mathscr{R}$ we define a one-sided,
  finite or infinite 
  sequence 
  $g_{0}, g_{1}, g_{2}, \ldots$
  of labels in $G_R$, called the {\it  cutting sequence} of $\gamma$ as follows (see \textsc{Figure}~\ref{fig:geo}):
  $g_0$ is the exterior label of the side of $R$ across which $\gamma$ crosses from $R$ to
  $g_0R$, 
  and for each $n\geq1$ we use  $g_n$
  to denote the exterior label of the side of 
  $g_0\cdots g_{n-1}R$
  across which $\gamma$ crosses from 
  $g_0\cdots g_{n-1}R$ to $g_0\cdots g_{n}R$. 

Given a discrete set $S$ and
 a set $Z$ of one-sided infinite sequences
 $(z_n)_{n=0}^\infty=z_{0}z_{1}\cdots$ in the cartesian product topological space $S^{\mathbb N}$, let
$E(Z)$ denote the set of finite words in $S$ that appear in some element of $Z$. 
For an integer $n\geq1$, let $E^n(Z)$
denote the set of elements of $E(Z)$ with word length $n$.

A word $w\in E(G_R^{\mathbb N})$
represents the group element given by the combination of the symbols under the group operation. From now on, the word length of elements of $G$ is always understood with respect to $G_R$.
We say $w$
is {\it  shortest} if its word length is equal to the word length of the element of $G$ represented by $w$, and we say $w$ is
 {\it  reduced} if it does not contain successive letters $e,\bar e\in G_R$. 
Shortest words are reduced.
We say 
$(g_n)_{n=0}^\infty\in G_R^{\mathbb N}$ is {\it shortest} if
  $g_j \cdots g_k$ is shortest, for all $j$, $k\in\mathbb N$ with $j<k$.
\begin{prop}[\cite{Ser86}, Theorem~3.1(ii)] \label{cutting-is-shortest} If $R$ is admissible, then the cutting sequences of $\gamma \in \mathscr{R}$  are shortest. \end{prop}

 \begin{figure}
\begin{center}
\includegraphics[height=7.5cm,width=7.5cm]{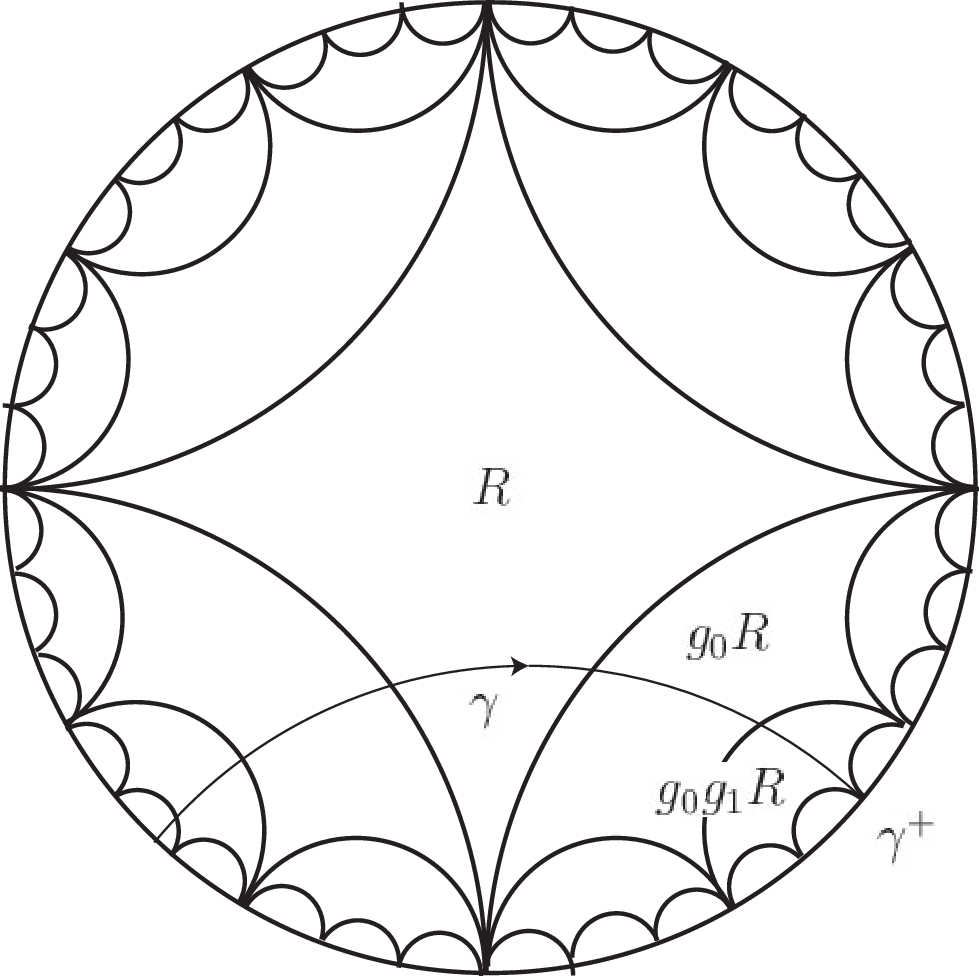}
\caption{An oriented geodesic $\gamma$ with the finite cutting sequence $g_0,g_1$ for a free Fuchsian group with two generators.}
\label{fig:cusp}
\end{center}
\end{figure}

 The cutting sequence of $\gamma\in\mathscr{R}$ may not always be infinite. \textsc{Figure}~\ref{fig:cusp} shows  an example  with $\gamma^+ \in \Lambda \setminus \Lambda_c$ for a group of the first kind. Note that $\gamma^+$ is the image of a cusp of $R$ under $G$ and $\gamma$ has no infinite cutting sequence. The next lemma characterizes $\gamma \in \mathscr{R}$ with  infinite cutting sequence.

\begin{lemma}\label{cut-converge}
 An element $\gamma\in\mathscr{R}$ has an infinite cutting sequence if and only if $\gamma^+\in \Lambda_c$. Moreover, for $\gamma\in \mathscr{R}$ with an infinite cutting sequence 
  $(g_n)_{n=0}^{\infty}$ we have  \[\lim_{n\to\infty}g_0\cdots g_n0=\gamma^+.\]
\end{lemma}
\begin{proof}

First assume that $\gamma$ has an infinite cutting sequence $(g_n)_{n=0}^\infty$.
Since cutting sequences are shortest by Proposition~\ref{cutting-is-shortest},  
 $(g_0\cdots g_n)_{n=0}^\infty\subset G$ are pairwise distinct. Since $R$ is locally finite, 
${\rm diam}( g_0\cdots g_nR) \rightarrow 0$
as $n\rightarrow \infty$. Hence, 
$g_0\cdots g_n0\rightarrow \gamma^+$ 
and therefore, $\gamma^+\in \Lambda$. To prove that $\gamma^+\in \Lambda_c$, we assume for the sake of contradiction that $\gamma^+$ is  fixed by  some  parabolic element of $G$. By \cite[Corollary~9.2.9]{Bea83}, $\gamma^+$ is the $G$-image of some cusp of $R$. This implies that $\gamma$ has a finite cutting sequence, and gives the desired contradiction.  
Hence, $\gamma^+\in \Lambda_c$.

Conversely,  assume that  $\gamma \in \mathscr{R}$ with $\gamma^+\in \Lambda$ has no infinite cutting sequence. In this case, $\gamma^+$ belongs to the Euclidean boundary of some image of $R$ under $G$. Hence, by \cite[Theorem~9.3.8]{Bea83},  $\gamma^+$ is fixed by some parabolic element of $G$ and therefore,  $\gamma^+ \notin \Lambda_c$. The proof is complete.
\end{proof}

\begin{figure}
\begin{center}
\includegraphics[height=11cm,width=11cm]{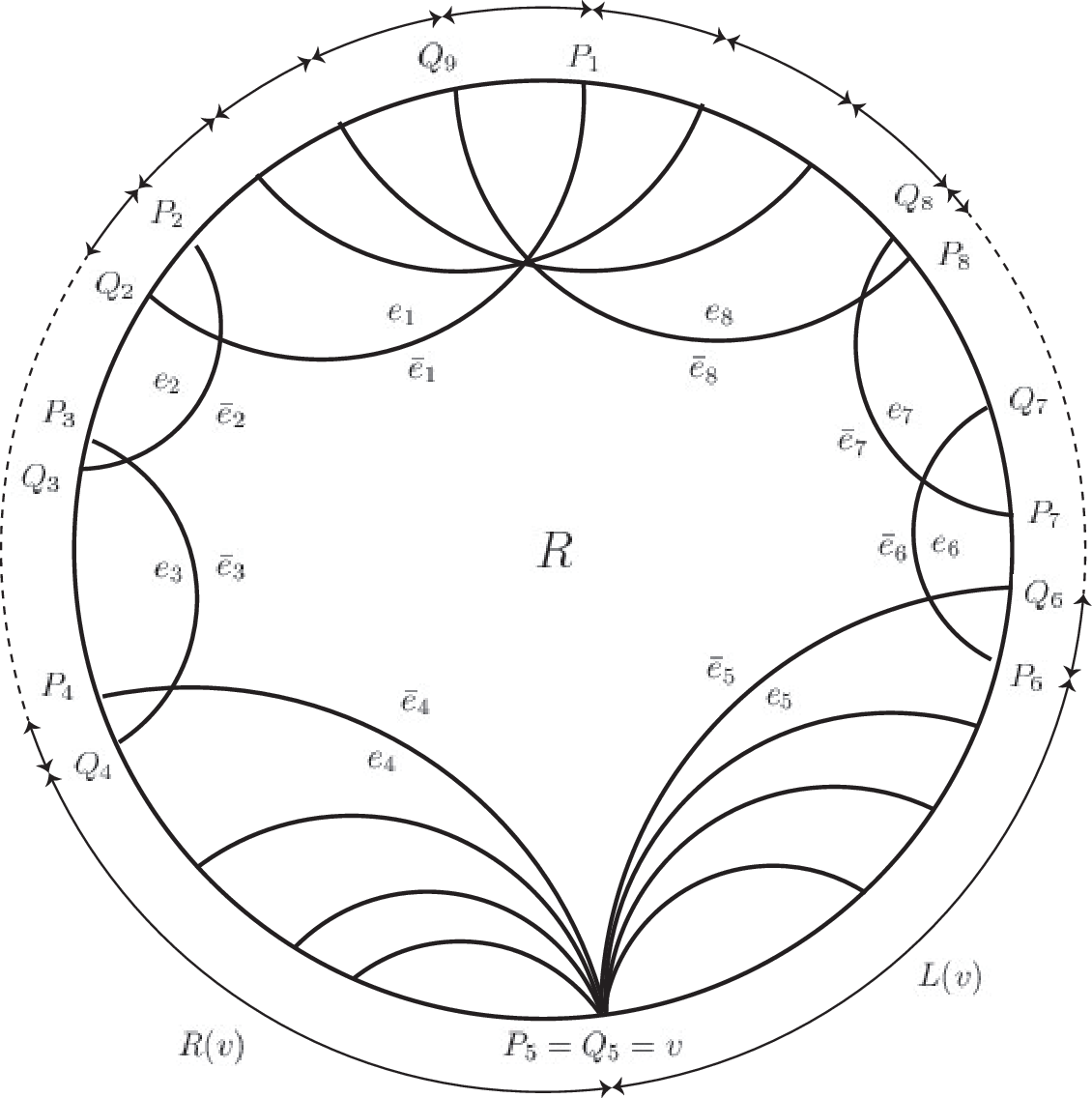}
\caption{A fundamental domain $R$ of a finitely generated Fuchsian group of the first kind with eight sides:
 $e_1$ and $e_8$,  $e_2$ and $e_6$, $e_3$ and $e_7$, $e_4$ and $e_5$ are identified in pairs,
which yields a hyperbolic surface of genus $2$. The bidirectional arrows indicate elements of the finite Markov partition constructed in Section~\ref{markov-sec}.}
\label{fig:BS}
\end{center}
\end{figure}

\subsection{The definition of the Bowen-Series map}\label{BowSer}
 Let $m$ denote the number of sides of the fundamental domain $R$, with exterior labels 
 $e_1,\ldots,e_m$ in anticlockwise order. 
For $1\leq i\leq m$ let $C(\bar e_i)$  denote 
the Euclidean closure of the  geodesic that contains the side of $R$ with the exterior label $e_i$.
   We denote the two endpoints of $C(\bar e_i)$ by $P_i$ and $Q_{i+1}$ in anticlockwise order (see \textsc{Figure}~\ref{fig:BS}).
If 
$C(\bar e_i)\cap C(\bar e_{i+1})\neq\emptyset$ we put $U_{i+1}=P_{i+1}$, and put $U_{i+1}=Q_{i+1}$ otherwise. For $j\in\mathbb Z$ with $i=j$ mod~$m$, set $e_j=e_i$, $P_j=P_i$,  $Q_{j}=Q_{i}$ and $U_j=U_i$.   We define
\[
\varDelta= \mathbb S^1 \setminus \bigcup_{i=1}^{m} [U_{i},P_{i}).
\]
Note that  $\varDelta=\mathbb S^1$ if $G$ is of the first kind.
According to\footnote{In \cite{BowSer79}, the Bowen-Series map is defined only for groups of the first kind.} \cite{BowSer79,Ser86},
  the Bowen-Series map $f\colon\varDelta\to \mathbb S^1$ is given by
\[f|_{[P_i,U_{i+1})}(\xi)=\bar e_i\xi.\]
 The {\it  $f$-expansion} of a point 
  $\xi\in \bigcap_{n=0}^\infty f^{-n}(\varDelta)$
is the one-sided infinite sequence $\xi_f=(e_{i_n})_{n=0}^{\infty} \in G_R^{\mathbb N}$ given by
\begin{equation}\label{BS-second}
f^n(\xi)\in 
[P_{i_n},U_{i_{n}+1})\quad\text{for } n\ge 0.
\end{equation}
We set \[\Sigma^+=\{\xi_{f}\colon \xi\in\Lambda\}.\]
For each $i\in\mathbb Z$, the restriction of $f$ to
$(P_i,U_{i+1})$
is analytic and can be extended to a $C^\infty$ map on $[P_i,U_{i+1}]$.
The derivatives of $f$ at points $P_i$ (resp. $U_{i+1}$) are the right-sided
(resp. left-sided) derivatives.
If $P_i$ (resp. $U_{i+1}$) is a cusp, it is a neutral periodic point of $f$.

\medskip
\noindent{\bf  (Standing hypothesis for the rest of the paper)}: $R$ is an admissible fundamental domain for $G$, and $f$ is the associated Bowen-Series map. 


\subsection{Characterization of admissible BS words}\label{character}
If $v$ is a vertex of $N$ in $\mathbb D$, 
let $n(v)$ denote the number of sides of $N$ through $v$. A small circle around $v$ has a cutting sequence 
$g_1 \cdots g_{2n(v)}$, and $g_1\cdots g_{2n(v)}=1$ is one of the defining relations of $G$.
Note that the relator has even word length since $R$ has even corners. 
A word $w\in E^k(G_R^{\mathbb N})$ is a {\it  clockwise} (resp. {\it  anticlockwise}) {\it  cycle} around $v$ if $k\leq 2n(v)$ and there exists a neighborhood $U$ of $v$ in $\mathbb D$ such that $w$ appears in the ``cutting sequence'' of 
any clockwise (resp. anticlockwise) circle around $v$ in $U$. 
  If moreover $k=n(v)$, we call $w$  a {\it  half cycle},
and  if $k>n(v)$ we call $w$ a {\it  long cycle}.

\begin{prop}[\cite{Ser86}, Theorem~4.2]\label{ser-lem}
A word in $E(G_R^{\mathbb N})$
is contained in $E(\Sigma^+)$ if and only if it is shortest and contains no anticlockwise half cycle.
\end{prop}

\subsection{Uniform decay of BS cylinders}\label{no-attr}
  Let $n\geq1$ and let  $e_{i_0}\cdots e_{i_{n-1}}\in E^n(\Sigma^+)$.
We define a {\it  BS cylinder}, or more precisely
 a BS $n$-cylinder by 
 \[\Theta(e_{i_0}\cdots e_{i_{n-1}})=\left\{\xi\in \varDelta\colon
f^k(\xi)\in [P_{i_{k}},U_{i_{k}+1} )\quad\text{for } 0\leq k\leq n-1  \right\}.\]
 Any BS cylinder is a left-closed and right-open arc. In what follows, we denote elements of  $E^n(\Sigma^+)$ by $a_0\cdots a_{n-1}$, $a_k\in G_R$ for $0\leq k\leq n-1$, 
 to make a distinction from cutting sequences of geodesics.
Put
\[\Theta_{\max,n}=\max_{a_0\cdots a_{n-1}\in E^n(\Sigma^+)}|\Theta(a_0\cdots a_{n-1})|.\]

If $G$ has no parabolic element, then
$\Theta_{\max,n}$ decays as $n$ increases
since some power of $f$ is uniformly expanding \cite[Theorem~5.1]{Ser81b}. 
Below we show that this uniform decay of BS cylinders still holds even if $G$
has a parabolic element.  
Although some results in \cite{Ser81b} seem to imply this, we give a self-contained proof for the convenience of the readers.
For $e\in G_R$ we denote by $H(\bar{e})$  the open  half-space in $\mathbb{D}$ bordered  by $C(\bar{e})$ which does not contain $R$.

\begin{lemma}\label{non-back}
 Let $(a_n)_{n=0}^{\infty}\in \Sigma^+$ and $n\ge 0$. We have $a_0\cdots a_n0\notin a_0\cdots a_{n}H(\overline a_{n+1})$, and for $k>n$ we have  $a_0\cdots a_k0\in a_0\cdots a_{n}H(\overline a_{n+1})$.
\end{lemma}

\begin{proof}  Clearly, $a_0\cdots a_{n}0 \notin a_0 \cdots a_n H(\overline a_{n+1})$ and $a_0\cdots a_{n+1}0 \in H(\overline a_{n+1})$. 
Since by Proposition~\ref{ser-lem} the elements of $E(\Sigma^+)$ are shortest,
$a_0\cdots a_{k}0 \in H(\overline a_{n+1})$  for $k>n$.  Hence, the lemma follows.
\end{proof}


\begin{lemma} We have
    \label{shrinking}
\[ \lim_{n\rightarrow \infty} \max_{a_0\cdots a_{n-1}\in E^n(\Sigma^+)}|\mathbb{S}^1\cap a_0\cdots a_{n-1}H(\overline a_n )|=0.\]    
In particular, we have
$\lim_{n\to\infty}\Theta_{\max,n}=0.$
\end{lemma}
\begin{proof}
Recall that each $a_0\cdots a_{n-1} \in E^n(\Sigma^+)$ has word length $n$ by Proposition~\ref{ser-lem}. 
For convenience we work in the upper half-plane  $\mathbb{H}$. We choose a conjugacy which maps a point in the  complement of the Euclidean closure of  the arc cut off by  $H(\overline a_0)$ in $\mathbb{S}^1$ to infinity.  
Put $r_n=\max_{a_0\cdots a_{n-1}\in E^n(\Sigma^+)} \textrm{diam}( a_0\cdots a_{n-1}R)$. Since $R$ is locally finite, we have $r_n\to0$ as $n\to\infty$.

If $C(\overline a_n)$ is a free side of $R$, then  $|\partial \mathbb{H}\cap a_0\cdots a_{n-1}H(\overline a_n) | \le r_n$.  If $C(\overline a_n)$  is not a free side, we  assume for simplicity that the side $s$ of $a_0 \cdots a_{n-1} R$ contained in $a_0 \cdots a_{n-1} C(\overline a_n)$ has one  vertex $v$ at infinity, and one vertex $v'$ in $\mathbb{H}$. Denote the other side of $a_0 \cdots a_{n-1} R$ emanating from the vertex $v'$ by $s'$, and  the endpoint of $s'$ not equal to $v'$ by $v''$. Denote by $s''$ the side of $a_0 \cdots a_{n-1} R$ emanating from the vertex $v''$ not equal to $s'$. Since the circular arcs containing $s$ and $s''$ are  disjoint \cite[Lemma~2.2]{BowSer79}, it is easy to see that  $|\partial \mathbb{H}\cap a_0\cdots a_{n-1}H(\overline a_n)| $ is bounded from above by the Euclidean distance of $v$ and $v''$. Since $v$ and $v''$ are vertices of the fundamental domain $a_0 \cdots a_{n-1} R$, we conclude $|\partial \mathbb{H}\cap a_0\cdots a_{n-1}H(\overline a_n)|\le r_n $.  The remaining cases can be treated in a similar fashion.   
The proof of the first assertion is complete. The second assertion follows from the first because 
$\Theta(a_0\cdots a_{n})$ is contained in the Euclidean closure of $a_0\cdots a_{n-1}H(\overline a_n)$.
\end{proof}

\subsection{Comparison of BS and cutting orbits}\label{conv-cut}
 Let $\gamma\in\mathscr{R}$ with
 $\gamma^+\in\Lambda$.
 Since $\Lambda$ is $G$-invariant and $\Lambda\subset \varDelta$,
 we obtain $\Lambda\subset \bigcap_{n=0}^{\infty} f^{-n}(\varDelta)$. Hence, $\gamma^+$ has an infinite $f$-expansion.  Let 
  $(a_n)_{n=0}^{\infty}$ denote the $f$-expansion of $\gamma^+$.
We call $(a_0\cdots a_n0)_{n=0}^\infty$ a {\it BS orbit} associated with $\gamma$. For BS orbits, the convergence is uniform in the following sense.

\begin{lemma}\label{a_nconverge}
 For any $\varepsilon>0$ there exists $n_0\geq1$ such that
if $\xi\in \Lambda$ has the $f$-expansion $(a_n)_{n=0}^{\infty}$, then for all $n\geq n_0$ we have
\[|a_0\cdots a_{n-1}0-\xi|<\varepsilon.\]
\end{lemma}

\begin{proof} 
By Lemma~\ref{non-back} we have 
$a_0\cdots a_{n-1}0\in a_0\cdots a_{n-2}H(\overline a_{n-1})$. By Lemma~\ref{shrinking}, $\textrm{diam}(a_0\cdots a_{n-2}H(\overline a_{n-1}))$ tends to zero uniformly, as $n\rightarrow \infty$. Since $\xi$ has the $f$-expansion  $(a_n)_{n=0}^{\infty}$, it belongs to the Euclidean closure of $a_0\cdots a_{n-2}H(\overline a_{n-1})$ for each $n$. The lemma follows.
\end{proof}


 Let $(g_n)_n$ denote the finite or infinite cutting sequence of $\gamma$.
 We call $(g_0\cdots g_n0)_n$ the {\it  cutting orbit} associated with $\gamma$.  For free groups, the cutting orbit of $\gamma \in \mathscr{R}$  coincides with the $f$-expansion of $\gamma^+$. For non-free groups, this is not always the case.
  Nevertheless, they differ only slightly in the sense of the next lemma.
  
  \begin{lemma}\label{lem-parallel}
  Let $\gamma\in\mathscr{R}$ have the infinite cutting sequence 
  $(g_n)_{n=0}^\infty$ and let $(a_n)_{n=0}^\infty$ be the $f$-expansion of $\gamma^+$.  For any $n\geq0$,
 $g_0\cdots g_nR$ and 
$a_0\cdots a_{n}R$ share a common side of $N$, or else
share a common vertex of $N$ in $\mathbb D$.
\end{lemma}
\begin{proof}
By Lemmas~\ref{cut-converge} and \ref{a_nconverge}, the cutting orbit 
 $(g_0\cdots g_n0)_{n=0}^\infty$
and
the BS orbit $(a_0\cdots a_{n}0)_{n=0}^\infty$ converge to the same point $\gamma^+$.
Hence,
the conclusion is a consequence of  \cite[Proposition~3.2]{Ser86}
and \cite[Theorem~3.1]{Ser86}.
\end{proof}

\subsection{Mild distortion on BS cylinders}\label{milddist}
For $n\geq1$ define
\[D_{n}=\sup_{a_0\cdots a_{n-1}\in E^n(\Sigma^+)}\sup_{\xi,\eta\in \Theta(a_0\cdots a_{n-1} )}\frac{|(f^{n})'\xi|}{|(f^{n})'\eta|}.\]
If $G$ has no parabolic element, some power of $f$ is uniformly expanding \cite[Theorem~5.1]{Ser81b}, and so $D_{n}$ is uniformly bounded. If $G$ has a parabolic element, $D_{n}$ grows as $n$ increases but sub-exponentially, which suffices for all our purposes. We say $f$ has {\it mild distortion} if $\log D_{n}=o(n)$ $(n\to\infty)$.
 \begin{prop}\label{MILD}
The Bowen-Series map $f$ has mild distortion.
\end{prop}
\begin{proof}
 Let $n\geq2$ and $a_0\cdots a_{n-1}\in E^n(\Sigma^+)$. By the chain rule and the mean value theorem for $\log|f'|$, for 
$\xi,\eta\in \Theta(a_0\cdots a_{n-1})$ we have
\[\begin{split}\log\frac{|(f^{n})'\xi|}{|(f^{n})'\eta|}&\ll 
\sum_{j=0}^{n-1}|f^j(\Theta(a_0\cdots a_{n-1}))|
\ll 
\sum_{j=0}^{n-2}\Theta_{\max,n-j}+2\pi,\end{split}\]
which is $o(n)$ by
Lemma~\ref{shrinking}.
\end{proof}

\subsection{Decay estimate of BS cylinders}\label{decayest}
The next proposition connects the size of a BS $n$-cylinder with the corresponding growth rate. 
There exists a constant  
$\theta_{0}>0$  such that
for $n\geq1$ and $a_0\cdots a_{n-1}\in E^n(\Sigma^+)$,     \begin{equation}\label{lower-bound}0<\theta_0\leq |\overline {a_{0}\cdots a_{n-1}}\Theta(a_0\cdots a_{n-1})|\leq 2\pi.
  \end{equation}

Put 
\[  n(R)=\max(\{n(v)\colon \text{$v$ is a vertex of $R$ in $\mathbb D$}\}\cup\{1\}).\]
\begin{prop}\label{z-series}
For any $\gamma\in\mathscr{R}$ 
with the cutting sequence 
  $(g_n)_{n=0}^\infty$ and the $f$-expansion $(a_n)_{n=0}^{\infty}$ of $\gamma^+$,
we have
\[1
\ll\frac{|\Theta(a_0\cdots a_{n-1})|}{\exp(-t_n(\gamma)  )}\ll
D_n.\]
\end{prop}

\begin{proof}

By Lemma~\ref{lem-parallel},
the copies
$g_0\cdots g_{n-1}R$ and
$a_0\cdots a_{n-1}R$ of $R$ share a common side of $N$, or else
share a common vertex of $N$ in $\mathbb D$.
The triangle inequality yields
\[|t_n(\gamma)-d(0,a_0\cdots a_{n-1}0)|\leq  n(R)\max\{d(0,g0)\colon g\in G_R\} \ll1.\]
Hence, it suffices to show that for all $n\geq1$ and $a_0\cdots a_{n-1}\in E^n(\Sigma^+)$,
\begin{equation}\label{z-series-eq}1
\ll\frac{|\Theta(a_0\cdots a_{n-1})|}{\exp(-d(0,a_0\cdots a_{n-1}0))}\ll
D_n.\end{equation}

Let $n\geq1$ and let $a_0\cdots a_{n-1}\in E^n(\Sigma^+)$.
Let  $\xi_+$ and $\xi_-$ denote the boundary points of $\Theta(a_0\cdots a_{n-1})$.
Let  $\theta>0$ denote the angle between the geodesic arcs joining $a_0\cdots a_{n-1}0$ to
 $\xi_+$ and $\xi_-$.
Since all $a_0,\ldots, a_{n-1}$ are M\"obius transformations, \eqref{lower-bound} gives
\begin{equation}\label{angle}\theta_{0}\leq \theta\leq 2\pi .\end{equation}

Split $\Theta(a_0\cdots a_{n-1})$ into three disjoint arcs
 $\Theta_+$,
$\Theta_0$,
$\Theta_-$
so that $\xi_+\in \Theta_+$, 
$\xi_-\in \Theta_-$ 
and the
$\overline{a_{0}\cdots a_{n-1}}$-images of the three arcs have the same Euclidean lengths.
We use 
$\Theta_+$ and $\Theta_-$ as a buffer, and estimate $|\Theta_0|$
rather than $|\Theta(a_0\cdots a_{n-1})|$ itself.
The mean value theorem gives
\begin{equation}\label{Z-ineq}\frac{1 }{3D_n}\leq\frac{|\Theta_0 |}{|\Theta(a_0\cdots a_{n-1})|}\leq\min\left\{\frac{D_n}{3},1\right\}.\end{equation}
 Let $d(0,a_0\cdots a_{n-1}0)=r$. Rotate the Poincar\'e disk so that $a_0\cdots a_{n-1}0$ is placed on the negative part of the real axis.
 By Lemma~\ref{a_nconverge}, 
  there exists $n_0\geq1$ such that if $n\geq n_0$, then for any $a_0\cdots a_{n-1}\in E^{n}(\Sigma^+)$,
 $\Theta(a_0\cdots a_{n-1})$ is contained in
 the Euclidean open ball of radius $1/100$ about $-1$. In particular, 
 $\Theta(a_0\cdots a_{n-1})$ does not contain $1$.
 We apply the M\"obius transformation 
 $T\colon\mathbb P^1\to\mathbb P^1$ given by
\[T(z)=\frac{\cosh\frac{r}{2}z+\sinh\frac{r}{2}}{\sinh\frac{r}{2}z+\cosh\frac{r}{2}}.\]
This carries the four geodesics through $a_0\cdots a_{n-1}0$ separating $\Theta_+$, $\Theta_0$,
$\Theta_-$
to rays through $0$
at an equal angle $\theta/3$.
Since $1\notin \Theta(a_0\cdots a_{n-1})$ and $T(1)=1$,
$1\notin T(\Theta(a_0\cdots a_{n-1}))$ holds.
Therefore,
$T(\Theta_0)$ lies in the complement of the domain $\{z\in\mathbb D\colon |{\rm arg}(z)|\leq\theta_0/3\}$.
A  calculation shows
\[|(T^{-1})'z|=\left|\cosh\frac{r}{2}-{\rm Re}(z)\sinh\frac{r}{2}-\sqrt{-1}{\rm Im}(z)\sinh\frac{r}{2}\right|^{-2}.\]
Since $T(\Theta_0)$ is uniformly bounded away from $1$ in the Euclidean distance,
we have
$|(T^{-1})'z|\asymp e^{-r}$.
Since
$|\Theta_0|=\int_{T(\Theta_0)}|(T^{-1})'z||dz|$, this yields
\begin{equation}\label{Z-ineq2}|\Theta_0|\asymp \theta e^{-r}.\end{equation}
Combining \eqref{angle}, \eqref{Z-ineq} and \eqref{Z-ineq2} we obtain the desired double inequalities.
\end{proof}

\subsection{Equality of level sets, boundary of the $\mathscr{H}$-spectrum}\label{L=H}
The upper and lower pointwise Lyapunov exponents at a point $\xi\in\Lambda$ are given by
\[\overline\chi(\xi)=\limsup_{n\to\infty}\frac{1}{n}\log|(f^n)'\xi|\quad\text{and}\quad\underline{\chi}(\xi)=\liminf_{n\to\infty}\frac{1}{n}\log|(f^n)'\xi|,\]
respectively.
If $\overline\chi(\xi)=\underline{\chi}(\xi)$, this common value is called the {\it pointwise Lyapunov exponent} at $\eta$ and denoted by $\chi(\xi)$. For each $\alpha\in\mathbb R$,
define the level set
\[
\mathscr{L}(\alpha)=\left\{\xi\in\Lambda_c\colon\overline\chi(\xi)=\underline{\chi}(\xi)=\alpha\right\}.\]The next proposition indicates that
the level sets of homological growth rates and that of pointwise Lyapunov exponents coincide.
\begin{prop}\label{coincide-lem}
For every $\gamma\in\mathscr{R}$ such that $\gamma^+\in\Lambda_c$
and every $n\geq1$,
\begin{equation}\label{coin-lem-eq}
D_n^{-1}\ll
\exp(\log|(f^n)'\gamma^+|-t_n(\gamma))\ll
D_n^2.\end{equation}
In particular, for every $\alpha\geq0$, $\mathscr{H}(\alpha)=\mathscr{L}(\alpha).$
\end{prop}
\begin{proof}
Let $(a_n)_{n=0}^{\infty}$ denote the $f$-expansion of $\gamma^+$.
By the mean value theorem, there exists $\xi\in\Theta(a_0\cdots a_{n-1})$ such that $|(f^n)'\xi||\Theta(a_0\cdots a_{n-1})|=|f^n(\Theta(a_0\cdots a_{n-1}))|$. By \eqref{lower-bound} we have $|f^n(\Theta(a_0\cdots a_{n-1}))|\in[\theta_0,2\pi]$, and 
\[\theta_0D_n^{-1}|\Theta(a_0\cdots a_{n-1})|^{-1}\leq|(f^n)'\gamma^+|\leq 2\pi D_n|\Theta(a_0\cdots a_{n-1})|^{-1}.\]
This together with
Proposition~\ref{z-series} 
yields \eqref{coin-lem-eq}.
The rest of the assertions follows from
 Proposition~\ref{MILD}.
\end{proof}
\begin{remark}\label{coincide-rem}
By Proposition~\ref{coincide-lem}, the
level sets $\mathscr{H}(\alpha)$ are pairwise disjoint.
\end{remark}

 \begin{lemma}\label{alp-lem2}
We have $\alpha_-=0$ if and only if $G$ has a parabolic element.
\end{lemma}

 \begin{proof}
 If $G$ has no parabolic element,
some power of $f$ is uniformly expanding \cite[Theorem~5.1]{Ser81b}. Hence, we have $\alpha_->0$ by Proposition~\ref{coincide-lem}.
If $G$ has a parabolic element then $R$ has a cusp, which is a neutral periodic point of $f$.
By \cite{JaeTak20},
the set  $\left\{x\in\Lambda\colon\overline\chi(x)=\underline{\chi}(x)=0\right\}$
has positive Hausdorff dimension, while
$\Lambda\setminus\Lambda_c$  is a countable set. Hence we have
 $\mathscr{L}(0)\neq\emptyset$, and 
 $\mathscr{H}(0)\neq\emptyset$ by Proposition~\ref{coincide-lem}.
 Hence we obtain
 $\alpha_-=0$.
\end{proof}
 
\begin{remark}  In the definition of $t_n(\gamma)$, we may replace the cutting sequence of $\gamma$ by the $f$-expansion of the positive endpoint $\gamma^+$. By Lemma~\ref{lem-parallel} this does not change the level sets $\mathscr{H}(\alpha)$. 
\end{remark}

\section{Finite Markov structures}
In this section, starting with the definition of Markov maps in Section~\ref{markov-map},
we construct a finite Markov partition for the Bowen-Series map $f$
 in Section~\ref{markov-sec}
by slightly modifying the Markov partition constructed in \cite{BowSer79}. In Section~\ref{equality} we use  this finite Markov partition to  identify the maximal invariant set of the Markov map $f$ as the limit set of $G$.
In Section~\ref{gen-p} we introduce a geometric pressure function
 using the  free energy of $f$-invariant Borel probability measures, and show 
 that the generalized Poincar\'e  exponent coincides with the geometric pressure.
\subsection{Markov maps}\label{markov-map}
Let $S$ be a discrete 
set with $\#S\ge 2$.
   A {\it  Markov map} is a map $F\colon \Gamma\to \mathbb S^1$ such that the following holds:
    \begin{itemize}
 \item[(M0)]
   There exists a family $(\Gamma(a))_{a\in S}$ 
of pairwise disjoint arcs in $\mathbb S^1$ 
such that $\Gamma=\bigcup_{a\in S}\Gamma(a)$.

 \item[(M1)] For each $a\in S$, the restriction 
 $F|_{\Gamma(a)}$
extends to a $C^1$ diffeomorphism from ${\rm cl}(\Gamma(a))$ onto its image. 
 \item[(M2)] If $a,b\in S$ and $F(\Gamma(a))\cap \Gamma(b)$ 
 has non-empty interior,
 then $F(\Gamma(a))\supset \Gamma(b)$.
 
 \end{itemize}
 The family  $(\Gamma(a))_{a\in S}$ of arcs is called a {\it  Markov partition} of $F$.

 Condition (M2) determines a transition matrix $(M_{ab})$ over the countable alphabet $S$ by the rule $M_{ab}=1$ if $F(\Gamma(a))\supset \Gamma(b)$ and $M_{ab}=0$ otherwise.
 This transition matrix determines a countable topological Markov shift
 $Y=Y(F,(\Gamma(a))_{a\in S})$ by
\begin{equation}\label{m-shift}Y=\{y=(y_n)_{n=0}^{\infty}\in S^{\mathbb N}\colon  M_{y_ny_{n+1}}=1 \ 
\text{for } n \ge 0\}.\end{equation}
We endow $Y$
with the metric $d_{Y}(y,z)=
\exp(-\inf\{n\geq0\colon y_n\neq z_n\})$
where we set $e^{-\infty}=0$.
For $n\ge 1$ and $\omega_0\cdots \omega_{n-1}\in S^n$, write 
\begin{equation}\label{sym-cyl}[\omega_0\cdots\omega_{n-1}]=\{y\in   Y\colon y_k=\omega_k\ \text{for }0\leq k\leq n-1 \}.\end{equation}
Subsets of $Y$ of this form are called 
cylinders. The collection of all cylinders forms a base of the topology on $Y$.

  For $\omega\in S^m$ and $\kappa\in S^n$,
 write $\omega\kappa$ for
$\omega_0\cdots \omega_{m-1}\kappa_0\cdots \kappa_{n-1}\in S^{m+n}$. 
 For convenience, put $E^0=\{ \emptyset\}$,  $|\emptyset|=0$, and  $\omega\emptyset=\omega=\emptyset\omega$ 
 for  all  $\omega\in E(Y)$.
 The Markov map $F$ is {\it  finitely irreducible} \cite{MauUrb03} if there exists a finite subset  $\Xi$ of $E(Y)\cup E^0$  such that
for all $\omega,\kappa\in E(Y)$ there exists $\lambda\in \Xi$ such that $\omega\lambda \kappa\in E(Y)$. 

 The symbolic dynamics and the dynamics of $F$ are related by the coding map $\pi_Y\colon Y\to \mathbb S^1$ given by
  \begin{equation}\label{code-map}\pi_Y((y_n)_{n=0}^{\infty})\in 
  \bigcap_{n=1}^{\infty}{\rm cl}(\Gamma(y_0\cdots y_{n-1})),\end{equation}
 where 
\begin{equation}\label{cyl}\Gamma(y_0\cdots y_{n-1})=\bigcap_{k=0}^{n-1} F^{-k}(\Gamma(y_k)).\end{equation}
We shall always assume that the Markov map $F$ has {\it decay of cylinders} \cite{JaeTak20}, that is, the right-hand side in \eqref{code-map} is a singleton. 
  We will treat two Markov maps
  introduced in Sections~\ref{markov-sec} and \ref{c-ind}.

\subsection{Construction of a finite Markov partition for the Bowen-Series map}\label{markov-sec}
We recall the construction of a Markov partition for the Bowen-Series map carried out in \cite{BowSer79}. Our presentation of this is a slightly 
expanded version so as to include groups of the second kind. All lemmas quoted from \cite{BowSer79} below remain valid for groups of the second kind.

A point $v\in\mathbb S^1$ is a {\it  proper vertex} of $R$   at infinity if $v$ is the common endpoint of two sides of $R$. A point $v\in\mathbb S^1$ is called an {\it  improper vertex} of $R$ at infinity if $v$ is the common endpoint of a side and a free side of $R$.
A proper vertex  at infinity is also called a {\it  cusp}. 
The set of all cusps of $R$ is denoted by $V_c$. 
Note that each $v\in V_c$ is a fixed point of some parabolic element of $G$. Conversely, if $G$ has a parabolic element, then $V_c$ is non-empty. 
 Let $V$ denote the set of all vertices of $R$ in $\mathbb{D}\cup \mathbb{S}^1$.
 
For each vertex $v\in V$ we denote by $N(v)$ the set of geodesics in $N$ 
passing through $v$, and by $W(v)$ the set of points where the geodesics in  $N(v)$ meet $\mathbb{S}^1$. 
The set $\bigcup_{v\in V}W(v)$ is $f$-invariant \cite[Lemma~2.3]{BowSer79}
and hence induces a Markov partition
for $f$. This partition is an infinite partition 
if and only if $R$ has a cusp.

If $R$ has a cusp,
 $f$ is not finitely irreducible with respect to this Markov partition, and so results on the multifractal analysis in \cite{JaeTak20} are not directly applicable.
In order to make use of the results in \cite{JaeTak20}, 
we construct a coarser  Markov partition below
with respect to which $f$ becomes finitely irreducible.

Let $v\in V_c$. There exists  $i\in\mathbb Z$ such that $v$ is contained in $C(\bar e_{i-1})\cap C(\bar e_{i})$.  We denote the arcs of $\mathbb S^1$ cut-off by successive points
of $W(v)$ in anticlockwise order from   $Q_{i+1}$ to $Q_{i}=v$ as  $L_1(v),L_2(v),\ldots$,
and anticlockwise from  
$Q_{i+1}$ to $Q_{i}$ by
$R_1(v),R_2(v),\ldots$, as in \textsc{Figure}~\ref{fig:BS}.

For each $v\in V_c$ we define
\[L(v)=\overline{ \bigcup_{r\geq 2}L_r(v)}\quad \text{and}\quad R(v)=\overline{\bigcup_{r\geq 3}R_r(v)}.\]
 For $v\in V$ we define 
\[W'(v)=\begin{cases}W(v)&\text{if $v\notin V_c$,}\\
\partial L_1(v)\cup\partial L(v)
\cup\partial R_2(v)\cup\partial R(v)&\text{if $v\in V_c$,}\end{cases}\]
and
\[W'=\bigcup_{v\in V}W'(v).\]
Note that $W'$ is a finite subset of $\bigcup_{v\in V}W(v)$.  As in \cite[Lemma~2.3]{BowSer79}, one verifies the following lemma. 
\begin{lemma}\label{endpoints-Markov}
We have
$f(W')\subset W'$. 
\end{lemma}
 We define a partition of $\varDelta$ into arcs 
 with endpoints given by two consecutive points in $W'$.
We choose all partition elements to be left-closed and right-open, in anticlockwise order. 
We label the partition elements by integers in a finite subset $S$ of $\mathbb N$, and
denote the element labeled with $a\in S$ by $\varDelta(a)$. By the arguments in  \cite[Lemma~ 2.5]{BowSer79} we obtain the  following important fact. 

\begin{prop} \label{BS-Markov} The Bowen-Series map $f:\varDelta\rightarrow \mathbb{S}^1$ defines a finitely irreducible Markov map with a finite Markov partition 
$(\varDelta(a))_{a\in S}$.
\end{prop}

The Bowen-Series map $f$
determines by \eqref{m-shift} a finitely irreducible Markov shift 
\[X=X(f,(\varDelta(a))_{a\in S}).\]
 The left shift $\sigma:X\rightarrow X$ is given by 
$(\sigma x)_n=x_{n+1}$ for $n \ge 0$.
  By Lemma~\ref{shrinking}, 
 the coding map $\pi=\pi_X$ given
 by \eqref{code-map} is well defined and continuous.
 We have 
 \begin{equation}\label{conju}f\circ \pi=\pi\circ\sigma.\end{equation}

  
BS cylinders and the cylinders in $X$ are related as follows.  
For each $a_0\cdots a_{n-1}\in E^{n}(\Sigma^+)$, 
the corresponding BS $n$-cylinder $\Theta(a_0\cdots a_{n-1})$ is the union of finitely many
 $n$-cylinders in $X$, the number of which is  at most $2n(R)$. 
Conversely,
for each $\omega_0\cdots\omega_{n-1}\in E^{n}(X)$ there exists a unique element
$a_0\cdots a_{n-1}$ of $E^n(\Sigma^+)$
such that $\varDelta(\omega_0\cdots\omega_{n-1})\subset\Theta(a_0\cdots a_{n-1})$. 
For convenience, we will sometimes identify $\omega_0\cdots \omega_{n-1}$ with the M\"obius transformation $a_0\cdots a_{n-1}$ in $G$, and 
write
$\Theta(\omega_0\cdots\omega_{n-1})$
for $\Theta(a_0\cdots a_{n-1})$. 

\subsection{Identifying the maximal invariant set}\label{equality}
  The proposition below asserts
 that the maximal invariant set of $f$ coincides with the limit set of $G$.
This clearly holds for groups of the first kind, and is known 
for free groups of the second kind \cite[Lemma~2.2]{Ser86}.
\begin{prop} \label{BSlimitset}
We have
\[\Lambda=\bigcap_{n=0}^{\infty} f^{-n}(\varDelta)=\pi(X).\]
\end{prop}
\begin{proof}
If $G$ is of the first kind, then clearly all the three sets are equal to $\mathbb S^1$.
 Suppose $G$ is of the second kind. Then $\varDelta$ is the union of finitely many left-closed and right-open arcs.
Points in $\partial\varDelta\setminus\varDelta$ are 
improper vertices of $R$. Since
each improper vertex of $R$ is paired with 
another, it is easy to see that improper vertices of $R$
are not limit points, and in particular $\partial\varDelta\setminus\varDelta$ is not contained in $\Lambda$. Interior points of the complement of $\varDelta$ are not limit points, because no copy of $R$ can accumulate at such a point. We have verified that $\Lambda\subset \varDelta$. 

Since $\Lambda$ is $G$-invariant and $\Lambda\subset \varDelta$,
 we obtain $\Lambda\subset \bigcap_{n=0}^{\infty} f^{-n}(\varDelta)$.  To prove the equalities in the proposition, 
 we first show the next lemma.
 \begin{lemma}\label{included}
 We have
 $\bigcap_{n=0}^{\infty} f^{-n}(\varDelta)\subset \pi(X)$.
 \end{lemma}
 \begin{proof}
 Let $\xi\in \bigcap_{n=0}^{\infty} f^{-n}(\varDelta).$
 Define $x=(x_n)_{n=0}^\infty\in S^\mathbb N$ by $f^n(\xi)\in\varDelta(x_n)$. This is well defined since the elements $\varDelta(a)$, $a\in S$ of the Markov partition are pairwise disjoint. Since $f$ preserves orientation and the elements of the Markov partition
 are left-closed and right-open arcs, $x\in X$.
 Clearly we have $\xi\in\pi(x)$.
  \end{proof}

 To complete the proof of Proposition~\ref{BSlimitset},
 it remains to show $\pi(X)\subset \Lambda$. 
Since $f|_{\Lambda}$ is transitive by Proposition \ref{BS-Markov}, the periodic points of $\sigma$ are dense in $X$. 
Since  $\pi$ is continuous,
it suffices to show that 
for any $k\geq1$
and any fixed point $x=(x_n)_{n=0}^{\infty}\in X$ of $\sigma^k$,
$\pi(x)\in \Lambda$ holds.
Observe that the M\"obius transformation $\overline{x_0\cdots x_{k-1}}\in G$ 
satisfies $\overline{x_0\cdots x_{k-1}}(\pi(x))=\pi(x)$.
Since $x_0\cdots x_{k-1}$ is not the identity in $G$ by Proposition~\ref{ser-lem}, 
and since $\Lambda$ contains all fixed points of elements of $G\setminus \{1 \}$ in $\mathbb S^1$, we obtain
$\pi(x)\in\Lambda$.
\end{proof}

Let $Y$ be a topological space, $Y_0\subset Y$ and let
 $F\colon Y_0\to Y$ be a Borel map.
Let $\mathcal M(Y_0,F)$ denote
the set of Borel probability measures
on $\bigcap_{n=0}^\infty F^{-n}(Y_0)$ which are invariant under the restriction of $F$ to this set.
 For each $\mu\in\mathcal M(Y_0,F)$, let $h(\mu)$
 denote the measure-theoretic entropy of $\mu$ with respect to $F$.

  We will use the following correspondence of invariant measures
  on $X$ and $\Lambda$.
\begin{lemma}\label{meas-rep}
For any $\mu\in\mathcal M(\Lambda,f)$  
there exists  $\nu\in\mathcal M(X,\sigma)$ 
such that $\mu=\nu\circ\pi^{-1}$
and $h(\mu)=h(\nu)$. Conversely,
 for any $\nu\in\mathcal M(X,\sigma)$, the measure
$\mu=\nu\circ\pi^{-1}$ belongs to $\mathcal M(\Lambda,f)$ and satisfies
$h(\mu)=h(\nu)$. 
\end{lemma}
\begin{proof}
 The coding map $\pi$ is one-to-one except on 
 the preimage of the countable set
$B=\bigcup_{n=0}^{\infty} f^{-n}(\bigcup_{a\in S}\partial\varDelta(a))$ where it is at most two-to-one. Since $f$ preserves boundary points of the elements of the Markov partition, $f^{-1}(B)=B$ and so $\sigma^{-1}(\pi^{-1}(B))=\pi^{-1}(B).$

We have $f\circ \pi=\pi\circ\sigma$, and the restriction of $\pi$ to $X\setminus\pi^{-1}(B)$ has a continuous inverse.
 Hence, $\pi$ induces a measurable bijection between $X\setminus\pi^{-1}(B)$
and $\pi(X)\setminus B$. 
This and $\Lambda\subset\pi(X)$ in Proposition~\ref{BSlimitset} implies that 
 for any $\mu\in\mathcal M(\Lambda,f)$ with $\mu(B)=0$
 there exists $\nu\in\mathcal M(X,\sigma)$ 
such that $\mu=\nu\circ\pi^{-1}$ and $h(\mu)=h(\nu)$.

If $\mu\in\mathcal M(\Lambda,f)$ and $\mu(B)>0$, there exist $\rho\in(0,1]$ and $\mu_1,\mu_2\in\mathcal M(\Lambda,f)$ such that $\mu_1(B)=0$, $\mu_2(B)=1$ and 
 $\mu=(1-\rho)\mu_1+\rho\mu_2$.  
  Since $B$ is a countable set,
$\mu_2$ is
supported on a periodic orbit of $f$.
By Proposition~\ref{BSlimitset}, there exists
$\nu_2\in\mathcal M(X,\sigma)$ that is supported on a periodic orbit of $\sigma$ and satisfies $\mu_2=\nu_2\circ\pi^{-1}$. By the previous paragraph,
there exists 
$\nu_1\in\mathcal M(X,\sigma)$ with
$\mu_1=\nu_1\circ\pi^{-1}$. Set $\nu=(1-\rho)\nu_1+\rho\nu_2$. Then
$\mu=\nu\circ\pi^{-1}$ and $h(\mu)=(1-\rho)h(\mu_1)=(1-\rho)h(\nu_2)=h(\nu)$
as required in the first assertion of the lemma. A proof of the second one is analogous.
\end{proof}




 \subsection{Equality of pressure and generalized Poincar\'e exponent}\label{gen-p}
The piecewise analytic function
$\phi\colon\Lambda\to\mathbb R$ given by
\[\phi=-\log|f'|\]
plays an important role.
 For $\mu\in \mathcal M(\Lambda,f)$, define the {\it  Lyapunov exponent} of $\mu$ 
 by $\chi(\mu)=-\int\phi d\mu.$
  The {\it  geometric pressure function}, or simply the pressure is the function
 $\beta\in\mathbb R\mapsto P(\beta\phi,f)$ given by \[P(\beta\phi,f)=\sup\{h(\mu)-\beta\chi(\mu)\colon\mu\in\mathcal M(\Lambda,f)\}.\]
A measure in $\mathcal M(\Lambda,f)$ which 
attains this supremum is called 
an {\it  equilibrium state} for the potential $\beta\phi$. 
By the affinity of 
entropy and Lyapunov exponent on 
 $\mathcal M(\Lambda,f)$, 
the geometric pressure function is convex.
It is non-increasing since any measure in $\mathcal M(\Lambda,f)$ has a non-negative Lyapunov exponent as 
in Lemma~\ref{negative} below.

\begin{lemma}\label{alp+-}
We have
\[\alpha_+=\sup\{\chi(\mu)\colon\mu\in\mathcal M(\Lambda,f)\}\quad
\text{and}\quad\alpha_-=\inf\{\chi(\mu)\colon\mu\in\mathcal M(\Lambda,f)\}.\]
\end{lemma}
\begin{proof}
 Using Proposition~\ref{MILD} and the irreducibility of the finite Markov shift $X$ in Proposition~\ref{BS-Markov}, we can construct a measure supported on periodic points whose Lyapunov exponent is arbitrarily close to $\alpha_+$. Hence, 
$\sup\{\overline{\chi}(\xi)\colon
\xi\in \Lambda_c\}\leq\sup\{\chi(\mu)\colon\mu\in\mathcal M(\Lambda,f)\}$ holds. The reverse inequality follows from Birkhoff's ergodic theorem. 
Combining this equality with 
$\alpha_+=\sup\{\overline{\chi}(\xi)\colon
\xi\in \Lambda_c\}$ which follows from
 Proposition~\ref{coincide-lem}, we obtain the first equality in the lemma.
 A proof of the second one is analogous.
\end{proof}

\begin{lemma}\label{negative}
For any $\mu\in\mathcal M(\Lambda,f)$ we have $\chi(\mu)\geq0$.
\end{lemma}
\begin{proof}
From Lemma~\ref{alp+-} and $\alpha_-\geq0.$
\end{proof}

Although $\phi$ may have discontinuities, 
the function $\varphi\colon X\to\mathbb R$ given by
\begin{equation}\label{varphi}\varphi=\phi\circ\pi\end{equation}
is continuous. 
For $\beta\in\mathbb R$ the topological pressure of  the potential 
$\beta\varphi\colon X\to\mathbb R$ with respect to $\sigma$ is given by 
\[P(\beta\varphi,\sigma)= \lim_{n\to\infty}\frac{1}{n}\log\sum_{\omega\in E^n(X) }\sup_{[\omega]}\exp \left(\beta\sum_{k=0}^{n-1}\varphi\circ\sigma^k\right).\]
Since $\varphi$ is continuous, 
the variational principle holds:
\[P(\beta\varphi,\sigma)=\sup\left\{h(\nu)+\beta\int\varphi d\nu\colon\nu\in\mathcal M(X,\sigma)\right\}.\]
Since
$\sigma$ is expansive and $X$ is a subshift over the finite set $S$,
the entropy map is upper semicontinuous on  $\mathcal M(X,\sigma)$.
Since $\varphi$ is continuous and
$\mathcal M(X,\sigma)$ is compact with respect to the weak* topology,
this supremum is attained. By Lemma~\ref{meas-rep},  there is an equilibrium state for the potential $\beta\phi$.

\begin{prop}\label{p-equal}
For all $\beta\in\mathbb R$ we have
\[P(\beta)=P(\beta\varphi,\sigma)=P(\beta\phi,f).\] 
\end{prop}

Let $g\in G$. 
A {\it  shortest representation} of $g$ is a representation of $g$ that contains exactly $|g|$ generators in $G_R$. 
A shortest representation of $g$ is {\it  admissible} if it is contained in $E(\Sigma^+)$.
\begin{lemma}\label{represent-u}
Every $g\in G\setminus\{1\}$ has a unique admissible 
shortest representation.
\end{lemma}
\begin{proof}
 Let $g=e_{i_1}\cdots e_{i_{|g|}}$ be a shortest representation of $g$. We replace all anticlockwise half cycles in this representation by the corresponding clockwise half cycles, and obtain (possibly) another shortest representation  $g=e_{j_1}\cdots e_{j_{|g|}}$ that contains no anticlockwise half cycle. By Proposition~\ref{ser-lem}, $e_{j_1}\cdots e_{j_{|g|}}\in E(\Sigma^+)$ holds.

Let $g=e_{j_1}\cdots e_{j_{|g|}}$, $g=e_{k_1}\cdots e_{k_{|g|}}$ be two admissible shortest representations of $g$. Suppose $e_{j_1}\neq e_{k_1}$.
Then we have a relation $\bar e_{k_{|g|}}\cdots \bar e_{k_1}e_{j_1}\cdots e_{j_{|g|}}=1$. Since the vertex cycles  give a complete set of relations of $G$ and both representations of $g$ are shortest,
$\bar e_{k_{|g|}}\cdots \bar e_{k_1}e_{j_1}\cdots e_{j_{|g|}}$ contains a cycle that contains 
$\bar e_{k_1}e_{j_1}$. It follows that one of the two representations of $g$ contains an anticlockwise half cycle and this yields a contradiction since both representations are admissible.
Hence, we obtain $e_{j_1}=e_{k_1}$. Repeating this argument we obtain $e_{j_i}\neq e_{k_i}$ for $1\leq i\leq |g|$.
\end{proof}

\begin{proof}[Proof of Proposition~\ref{p-equal}]
For $n\geq1$ and $a_0\cdots a_{n-1}\in E^n(\Sigma^+)$,
let $E^n(X,a_0\cdots a_{n-1})$ denote the set of elements of $\omega$ in $E^{n}(X)$ such that $\varDelta(\omega)\subset \Theta(a_0\cdots a_{n-1})$. Clearly, 
\begin{equation}\label{multiplicity}1\leq\#E^n(X,a_0\cdots a_{n-1})\leq 2n(R).\end{equation}
 By \eqref{lower-bound}, for $x\in \Theta(a_0\cdots a_{n-1})$,
\begin{equation}\label{z-distortion}\theta_0D_n^{-1}\leq\frac{|\Theta(a_0\cdots a_{n-1})|}{|(f^{n})'x|^{-1}}\leq 2\pi D_n.\end{equation}
By Proposition~\ref{z-series}, \eqref{multiplicity} and \eqref{z-distortion}, 
there exists a constant $C\geq1$ such that for $\beta,t\in\mathbb R$ we have
\begin{equation}\label{multi-double}
\begin{split}
C^{-\beta}D_n^{-2\beta}(2\pi)^{-\beta}&\leq\frac{\sum_{\omega\in E^n(X,a_0\cdots a_{n-1})}\sup_{[\omega]}\exp(\beta\sum_{k=0}^{n-1}\varphi\circ\sigma^k)e^{-tn}}{\exp(-\beta d(0,a_0\cdots a_{n-1}0))e^{-tn}}\\
&\leq 
  2n(R) \cdot C^\beta D_n^{2\beta}\theta_0^{-\beta}.\end{split}\end{equation}
By Lemma~\ref{represent-u} and 
Proposition~\ref{ser-lem},
there is a one-to-one correspondence between $G\setminus\{1\}$ and $E(\Sigma^+)$.
Therefore,
rearranging the double inequalities \eqref{multi-double} and summing the result over all words in $ E^{n}(\Sigma^+)$, and then summing the result over all $n\geq1$, and then using \cite[Theorem~2.1.3]{MauUrb03} we obtain 
\[P(\beta\varphi,\sigma)\leq\inf\left\{t\in\mathbb R\colon\sum_{n=1}^\infty D_n^{2\beta}\sum_{g\in G,|g|=n}\exp(-\beta d(0,g0)-t|g|)<+\infty\right\},\]
 and $P(\beta\varphi,\sigma)\leq P(\beta)$. A similar reasoning shows the reverse inequality.
Lemma~\ref{meas-rep} implies
$P(\beta\varphi,\sigma)=P(\beta\phi,f)$. This completes the proof of Proposition~\ref{p-equal}.
\end{proof}

\section{Building induced expansion}

The aim of this section is to construct from the Bowen-Series map $f$
a uniformly expanding induced Markov map $\tilde f$. We construct the induced Markov map $\tilde f$ in Section~\ref{c-ind} as
a first return map to a large subset of $\varDelta$ which misses small neighborhoods of the cusps. 
Although this construction is essentially the same as in \cite{BowSer79}, in order to build a uniform expansion 
without assuming the non-contracting condition \eqref{NC}
we use a linear growth lemma on induced scale
(Lemma~\ref{g-linear}) that relies on a geometric ingredient
developed in Section~\ref{control}.
Finally in Section~\ref{exp-section} we verify the uniform expansion of $\tilde f$.

 \subsection{Construction of induced Markov map}\label{c-ind}
Let $f$ be the Bowen-Series map with the finite Markov partition $(\varDelta(a))_{a\in S}$ constructed in Section~\ref{markov-sec}.
Note that $\varDelta(a)\cap\Lambda\neq\emptyset$ for $a\in S$.
Define
 \[\tilde\varDelta=\varDelta\setminus\left( V_c\cup \bigcup_{v\in V_c}L(v)\cup R(v)\right).\]
 Note that $\tilde\varDelta$ is a non-empty set.
Define  $t\colon \tilde\varDelta\to\mathbb N$
by 
\begin{equation}\label{t-def}t(\xi)=\inf\{n\geq1\colon f^{n}(\xi)\in \tilde\varDelta
\}\end{equation}
Define the induced map
\begin{equation}\label{inducedmap}\tilde{f}: \tilde\varDelta \to \mathbb{S}^1,\quad \xi\mapsto f^{t(\xi)}(\xi),\end{equation}
and set
\begin{equation}\label{induce-eq10}\tilde\Lambda=\bigcap_{n=0}^{\infty} \tilde f^{-n} (\tilde\varDelta).\end{equation}
Replacing each $\varDelta(a)$, $a\in S$,  by the countably many cylinders on which $t$ is finite and constant, we obtain a countably infinite subset $\tilde S$ of $E(X)$ such that $\tilde \varDelta=\bigcup_{\tilde a\in \tilde S}\tilde \varDelta(\tilde a)$ and a Markov map
$\tilde{f}$ with a Markov partition 
$(\tilde \varDelta(\tilde a))_{\tilde a\in \tilde S}$. This determines by \eqref{m-shift} a countable Markov shift
\[\tilde X=\tilde X(\tilde f,(\tilde \varDelta(\tilde a))_{\tilde a\in \tilde S}).\]Note that each cylinder $\tilde \varDelta(\tilde a)$ has the form
$\tilde \varDelta(\tilde a)=\varDelta(a_1)\cap\{t=n\}\cap f^{-n}(\varDelta(a_2))$
for some $n\geq1$ and $a_1,a_2\in S$.

\subsection{Control of deviations of cutting orbits}\label{control}
Let $\gamma\in\mathscr{R}$  with the infinite  cutting sequence
$(g_n)_{n=0}^\infty$. 
 If $G$ has a parabolic element, $R$ has a cusp and the cutting orbit  
 $(g_0\cdots g_n0)_{n=0}^\infty$ may deviate from $\gamma$. 
 In this subsection we elaborate on uniform bounds on this deviation using $\tilde{f}$. 



 For $\xi\in\mathbb D$
and $A\subset \mathbb D$ we denote $d(\xi,A)=\inf \{d(\xi,\eta)\colon \eta\in A\}$.

 \begin{lemma}
\label{gamma-cor}
There exist $C_0>0$ and an integer $M_0\geq1$ such that if $\gamma\in\mathscr{R}$ satisfies $\gamma^+\in \Lambda_c$ 
 and $f^k(\gamma^+)\in \tilde\varDelta$ for some  $k\geq M_0$, then 
there exists  $n_*\in\{k-M_0,\ldots, k+M_0\}$
such that
   \[d(g_0\cdots g_{n_*}0,\gamma\cap g_0\cdots g_{n_*}R)\leq C_0,\]
where $(g_n)_{n=0}^\infty$ denotes the cutting sequence of $\gamma$.
\end{lemma}

\begin{proof}
Let  $C_0>0$ be so large that the  hyperbolic disc around $0$ of radius $C_0$ covers all of the intersection of  $R$ and  the Nielsen region  of $G$ \cite[Section~8.5]{Bea83}  except small neighborhoods of the cusps. Let  $M_0\ge 1$ be a large number to be determined later.

Suppose for a contradiction the assertion of the lemma fails. Then there exists    $\gamma \in \mathscr{R}$ with infinite cutting sequence  
 $(g_n)_{n=0}^\infty$  (see Lemma~\ref{cut-converge})
and there exists  $k\ge M_0$ such that $f^k(\gamma^+)\in \tilde\varDelta$  and,  for every $n\in\{k-M_0,\ldots, k+M_0\}$,
 \[d(g_0\cdots g_n0,\gamma\cap g_0\cdots g_nR)> C_0.\]  
This means that $\gamma$ performs a deep cusp excursion between the $(k-M_0)$th and the  $(k-M_0)$th crossing of fundamental domains and therefore, the cutting symbols 
 $g_{k-M_0},\ldots, g_{k+M_0}$ of $\gamma$  are given by the periodic sequence of labels of sides ending at one of the a cusps of $R$, say $v_0\in V_c$. We conclude by 
Lemmas~\ref{non-back} and \ref{lem-parallel} that  the partial BS orbit $(a_{0}\cdots a_{n}0)_{n\colon |n-k|\leq M_0-2}$
appears in the same order in the partial cutting orbit 
 $(g_0\cdots g_n0)_{n\colon |n-k|\leq M_0}$.  This implies that the first $M_0 -2$ symbols of the $f$-expansion of $f^k(\gamma^+)$ are given by the periodic sequence of sides ending at the cusp $v_0$. If $M_0$ is large enough depending on the prime periods of the cusps, this  implies $f^k(\gamma^+)\in L(v_0)\cup R(v_0)$ contradicting $f^k(\gamma^+) \in \tilde \varDelta$.
\end{proof}



\begin{prop}\label{prop:bdd-deviation}
There exists  $C>0$ such
that for all $n\ge1$  sufficiently large and $\tilde\omega_0\cdots \tilde\omega_{n-1}\in E^{n}(\tilde X)$,  and for all  $\gamma \in \mathscr{R}$ with $\gamma^{+}\in \tilde \varDelta( \tilde{\omega}_{0}\cdots\tilde{\omega}_{n-1})\cap \Lambda_c$, 
\[
d(\tilde{\omega}_{0}\cdots\tilde{\omega}_{n-1}0,\gamma)\le C.
\]
\end{prop}

\begin{proof} Let $C_0$ and $M_{0}$ denote the constants in Lemma~\ref{gamma-cor}.
 Let $n> M_0$. There exists  $k\geq n$  and $a_1\cdots a_k\in E^k(\Sigma^+)$ such
that $\tilde{\omega}_{0}\cdots\tilde{\omega}_{n-1} = a_1 \cdots a_k $ and $f^{k}(\gamma^{+})=(\tilde{f})^{n}(\gamma^{+})\in\tilde{\varDelta}$.
 By Lemma~\ref{gamma-cor} there exists $n_{*}\in\{k-M_{0},\ldots,k+M_{0}\}$ such that
 $d(g_{0}\cdots g_{n_{*}}0,\gamma\cap g_{0}\cdots g_{n_{*}}R)\leq C_0$. By the triangle inequality we have
\[
d(a_1\dots a_k0,\gamma) \le  2M_{0}\max_{g\in G_{R}}d(0,g0) +  d(a_{1}\cdots a_{n_{*}}0,g_0\cdots g_{n_{*}}0) + C_0.
\]
 Since the second term of the right-hand side does not exceed  $n(R)\max_{g\in G_R}\{d(0,g0)\}$ 
 by Lemma~\ref{lem-parallel},  the proposition follows.
\end{proof}

\subsection{Uniform expansion of the induced map}
\label{exp-section}

If the Fuchsian group $G$ has no parabolic element, that is, $G$ is convex cocompact, then the next lemma follows from  the Milnor-Swarc Lemma.
\begin{lemma}[Linear growth on induced scale]\label{g-linear}
There exists $\alpha_0>0$ such that for all sufficiently large $n\geq1$ and  every  $\tilde\omega_0\cdots \tilde\omega_{n-1}\in E^{n}(\tilde X)$ we have
\[d(0,\tilde\omega_0\cdots \tilde\omega_{n-1}0)\geq \alpha_0n.\]
\end{lemma}

\begin{proof}Let $C_0$ and $M_{0}$ denote the constants in 
Lemma~\ref{gamma-cor}. Let $n> M_0$  and  $\tilde\omega_0\cdots \tilde\omega_{n-1}\in E^{n}(\tilde X)$.
Let $\gamma \in \mathscr{R}$ such that  $\gamma^+ \in \Lambda_c\cap\tilde \varDelta(\tilde\omega_0\cdots\tilde\omega_{n-1})$ with the cutting sequence
$(g_j)_{j=0}^{\infty}$.
By Lemma~\ref{gamma-cor}, for  each $k\in\{M_0,\ldots,n-1\}$ we fix an integer  $j(k)$ such that $|j(k)-|\tilde\omega_0\cdots\tilde\omega_k||\leq M_0$ and 
\begin{equation}\label{linear-eq0}
d(g_0 \cdots g_{j(k)}0,\gamma\cap g_0 \cdots g_{j(k)}R)\leq C_0. \end{equation} 
We write all the distinct elements of the sequence
$j(M_0),\ldots,j(n-1)$ as  $j_1,j_2,\ldots,j_q$  in the increasing order with some   $q\ge (n-1-M_0)/(2M_0)$. For each $1\le k \le q$ there exists $p_k\in \gamma\cap g_{0}\cdots g_{j_k}R$ such that $d(g_{0}\cdots g_{j_k}0,p_k)\le C_0.$
Using \eqref{linear-eq0} and 
Lemma~\ref{lem-parallel} we derive the existence of a uniform constant $C'>0$ such that 
\[
d(0,\tilde\omega_0\cdots \tilde\omega_{n-1}0) \ge d(p_1,p_q) -C'.
\] 
Divide the geodesic segment from $p_1$ to $p_q$ into segments of  hyperbolic length $C_0$, with one shorter segment, say $\gamma_1,\ldots, \gamma_N$, for some $N\ge 1$.
By \eqref{linear-eq0}, for each  $1\leq \ell\leq q$, the orbit point 
$g_{0}\cdots g_{j_\ell}0$ is within the hyperbolic distance $C_0$ of  one of the geodesic segments $\gamma_1,\dots, \gamma_N$.

Partition the set $\{g_{0}\cdots g_{j_\ell}0\colon 1\leq \ell\leq q\}$ into subsets $O_1,\dots, O_N$ so that $d(O_k, \gamma_k) \le C_0$ for $1\le k\le N$.  
Since $G$ acts properly discontinuously on $\mathbb D$,
there exists an integer $M\geq1$ such that  $\# O_k\le M$ for $1\le k\le N$.
Hence, 
$N\ge q/M$.  Combining this with \eqref{linear-eq0} yields
\[d(0,\tilde\omega_0\cdots \tilde\omega_{n-1}0)\geq C_0\left(\frac{q}{M}-1\right)-C'\geq
C_0\left(\frac{n-1-M_0}{2M_0M}-1\right)-C'.\]
Hence, the lemma follows for  $\alpha_0=C_0/(3M_0M)$ and  sufficiently large $n$.
\end{proof}

\begin{prop} \label{z-series-induce}
There exists $\alpha_{0}>0$ such that for $n\ge1$  sufficiently large, 
\[
\inf_{\tilde{\omega}_{0}\cdots\tilde{\omega}_{n-1}\in E^{n}(\tilde{X})}\inf_{\xi\in{\tilde\varDelta}(\tilde{\omega}_{0}\dots\tilde{\omega}_{n-1})\cap\Lambda_c}|(\tilde{f}^{n})'\xi|\gg e^{\alpha_{0}n}
\]
and
\[
{\rm diam}(\tilde\varDelta(\tilde{\omega}_{0}\dots\tilde{\omega}_{n-1})\cap{\Lambda_c})\ll e^{-\alpha_{0}n}.
\]
\end{prop}
\begin{proof}
Let $C>0$ denote the constant in Proposition~\ref{prop:bdd-deviation}. Let  $n\geq1$ and  $\tilde\omega_0\cdots \tilde\omega_{n-1}\in E^{n}(\tilde X)$ and let  $\xi\in\tilde\varDelta(\tilde{\omega}_{0}\dots\tilde{\omega}_{n-1})\cap \Lambda_c$.
Let $\gamma\in\mathscr{R}$ be the ray through zero with $\gamma^{+}=\xi$.
By Proposition~\ref{prop:bdd-deviation} we have  for $n$ sufficiently large, 
\begin{equation}
d(\tilde{\omega}_{0}\dots\tilde{\omega}_{n-1}0,\gamma)\le C.\label{eq:bdd-dist}
\end{equation}
Since $\tilde{f}^{n}\left(\gamma^{+}\right)=\left(\tilde{\omega}_{0}\dots\tilde{\omega}_{n-1}\right)^{-1}\gamma^{+}$,
it follows from the well-known properties of the Poisson kernel \cite{Bea83} that 
\[\log|(\tilde{f}^{n})'\gamma^{+}|=d(0,p),\]
where $p\in\mathbb D$ denotes the point of intersection between $\gamma$ and the horocircle at $\gamma^{+}$ through
$\tilde{\omega}_{0}\dots\tilde{\omega}_{n-1}0$.
By \eqref{eq:bdd-dist} we have  
$
d(p,\tilde{\omega}_{0}\dots\tilde{\omega}_{n-1}0)\le2C
$
 and thus, 
\begin{equation}
|\log|(\tilde{f}^{n})'\gamma^{+}|-d(0,\tilde{\omega}_{0}\dots\tilde{\omega}_{n-1}0)|\le2C.\label{eq:bdd-dist2}
\end{equation}
The first assertion of the proposition now follows from Lemma~\ref{g-linear}. 

To prove the second assertion, first note that  the estimate in \eqref{eq:bdd-dist} remains intact if $\gamma\in\mathscr{R}$ is a ray through zero whose endpoint $\gamma^+$ is in between two points in $\tilde \varDelta(\tilde{\omega}_{0}\dots\tilde{\omega}_{n-1})\cap \Lambda_c$. Consequently, the first assertion of the proposition
also holds if $\xi$ is taken from the smallest arc in $\mathbb{S}^{1}$
containing $\tilde \varDelta(\tilde{\omega}_{0}\dots\tilde{\omega}_{n-1})\cap\Lambda_c$.
From this and the mean value theorem, the second assertion of the
proposition follows. 
\end{proof}

 \begin{lemma}\label{alp-lem1}We have $\alpha_-<\alpha_+$.
\end{lemma}
 \begin{proof}
 If $G$ has a parabolic element, then $\alpha_-=0$ by Lemma~\ref{alp-lem2},
 and $\alpha_+>0$ since
the induced Markov interval map $\tilde f$ is uniformly expanding by Proposition~\ref{z-series-induce}. If $G$ has no parabolic element, it follows from \cite[Corollary~11.3]{Lalley89} that $\varphi$ in \eqref{varphi} is 
 not cohomologous to a constant. 
 Since  $f$ is piecewise $C^2$ and
 some iterate of $f$ is uniformly expanding,
  $\varphi$ is H\"older continuous. By a standard argument \cite[Proposition~4.5]{Bow75} we conclude that $\alpha_-<\alpha_+$. 
\end{proof}
\section{Thermodynamic formalism and multifractal analysis}
In this section we implement the thermodynamic formalism and
the multifractal analysis for the Bowen-Series map.
In Section~\ref{tf-unique}
we establish the uniqueness of equilibrium states and  the analyticity of the geometric pressure function.
In Sections~\ref{dim-formulas} and
\ref{bow-sec}, we apply
 results in \cite{JaeTak20} to obtain formulas for the Hausdorff dimension of level sets and the limit set. In Section~\ref{dim-sec} we derive formulas for the $\mathscr{H}$-spectrum and its first-order derivative in terms of the pressure. In Section~\ref{proof-thma} we complete the proof of the Main~Theorem.

\subsection{Uniqueness of equilibrium states, regularity of pressure}\label{tf-unique}
The next proposition is a key ingredient for the proofs of our main results. The proof relies heavily on the existence of an induced system which is uniformly expanding (see Proposition~\ref{z-series-induce}). Except for this geometrical fact, the arguments are well known, and can be found 
 in \cite{KesStr04}, \cite[Section~8]{MauUrb03}, \cite{PesSen08} for example.
For the convenience of the reader we include a proof in Appendix~A. 
\begin{prop}\label{press} 
The Bowen-Series map $f$ satisfies all of the following.
\begin{itemize}
\item[(a)] 
For any 
 $\beta\in(-\infty,\beta_+)$
there exists a unique equilibrium state for the potential $-\beta\log|f'|$, denoted by $\mu_\beta$.
 We have $\beta_+=+\infty$ if and only if  $G$ has no parabolic element. 
\item[(b)] The geometric pressure function $P$ is analytic 
 on $(-\infty,\beta_+)$. 
\item[(c)] For all $\beta\in(-\infty,\beta_+)$,
$P'(\beta)=-\chi(\mu_\beta)$.
In particular, the function
$\beta\in(-\infty,\beta_+)\mapsto\chi(\mu_\beta)$ is analytic.
\end{itemize}
\end{prop}

\subsection{Dimension formula for level sets}\label{dim-formulas}
We recall a few relevant definitions in \cite{JaeTak20}.
A measure $\mu\in\mathcal M(\Lambda,f)$ is  {\it  expanding} if  $\chi(\mu)>0$. The {\it  dimension} of a measure $\mu\in\mathcal M(\Lambda,f)$ is
defined by \[\dim(\mu)=\begin{cases}
\displaystyle{\frac{h(\mu)}{\chi(\mu)}}&\ \text{ if $\mu$ is expanding},\\
0&\ \text{ otherwise.}\end{cases}\]
For an ergodic expanding measure $\mu$,  the dimension 
$\dim(\mu)$ is equal to the infimum of the Hausdorff dimensions of sets
with full $\mu$-measure (see e.g., \cite[Theorem~4.4.2]{MauUrb03}). In particular, $\delta_G\geq\dim(\mu)$ holds for any $\mu\in\mathcal M(\Lambda,f)$.

We say $f$ is {\it  saturated} if
   \begin{equation}\label{saturate}
\delta_G=\sup\{\dim(\mu)\colon\mu\in\mathcal M(\Lambda,f)\}.\end{equation}
If $G$ has no parabolic element,
it is known \cite{Bow79,Ser81b} that the supremum in \eqref{saturate} is attained by a unique element, and
in particular $f$ is saturated.
This unique measure is equivalent to the   normalized $\delta_G$-dimensional Hausdorff measure on $\Lambda$ \cite{Pat76,Sul79}.
The saturation is important because it ensures that the
the dimension formula in \cite[Main~Theorem]{JaeTak20} accounts for any level set of positive Hausdorff dimension. 
Even in case 
 $G$ has a parabolic element,
the saturation still holds,
 although there is no measure which attains the supremum in \eqref{saturate}.

\begin{prop}\label{saturation}
The Bowen-Series map $f$ is saturated.
\end{prop}
\begin{proof} 
The case where $G$ has no parabolic element has already been explained.
Suppose $G$ has a parabolic element.
If \eqref{NC} holds, then $f$ is a non-uniformly expanding, finitely irreducible Markov map in the sense of \cite{JaeTak20}. Since  $\Lambda\setminus\bigcup_{n=0}^\infty f^{-n}(V_c)$ is contained in $\bigcup_{n=0}^\infty f^{-n}(\tilde\Lambda)$ and $V_c$ is countable, we have  $\dim_{\rm H}(\Lambda) = \dim_{\rm H}(\tilde\Lambda)$. Hence, $f$ is saturated by 
\cite[Proposition~5.2(c)]{JaeTak20}.
Even if \eqref{NC} does not hold, we have shown in  Proposition~\ref{z-series-induce} 
 that some power of the induced Markov map $\tilde f$ is uniformly expanding.
 Hence, the argument in the proof of \cite[Proposition~5.2(c)]{JaeTak20} works almost
  verbatim to conclude that $f$ is saturated.\end{proof}

\begin{prop}\label{l-formula} 
The Bowen-Series map $f$ satisfies all of the following.
\begin{itemize}
    \item[(a)] We have $\mathscr{H}(\alpha)\neq\emptyset$
if and only if $\alpha\in[\alpha_-,\alpha_+]$.
\item[(b)] For all $\alpha\in[\alpha_-,\alpha_+]$ we have
\begin{equation}\label{L-dimension}b(\alpha)=\lim_{\varepsilon\to0}\sup\{\dim(\mu)\colon\mu\in\mathcal M(\Lambda,f),\ |\chi(\mu)-\alpha|<\varepsilon\}.\end{equation}
\item[(c)] For all $\alpha\in[\alpha_-,\alpha_+]\setminus \{0\}$ we have
\[b(\alpha)=\max\left\{\dim(\mu)
\colon\mu\in\mathcal M(\Lambda,f),\ \chi(\mu)=\alpha\right\}.\]
\end{itemize}
\end{prop}

\begin{proof}
Proposition~\ref{coincide-lem} gives
$\mathscr{H}(\alpha)=\mathscr{L}(\alpha)$, and so $b(\alpha)=\dim_{\rm H}\mathscr{L}(\alpha)$.
If $G$ has no parabolic element, then some power of $f$ is uniformly expanding \cite[Theorem~5.1]{Ser81b}, and so the result is well known, see for example \cite{Ols03,Pes97,PesWei97,PesWei01,Wei99}.

Suppose $G$ has a parabolic element. The assertion in (a) follows from \cite[Main~Theorem(a)]{JaeTak20}.
To derive the desired formula in (b), we aim to apply \cite[Main~Theorem(b)]{JaeTak20}. By Proposition~\ref{BS-Markov},  $f$ is a
 finitely irreducible Markov map. By Proposition~\ref{MILD},  $f$ has  mild distortion, and by Proposition~\ref{saturation}, $f$ is saturated.
 In addition to these conditions, 
 in \cite[Main~Theorem(b)]{JaeTak20} it is  assumed that the map satisfies 
 the non-contracting condition as in \eqref{NC}. 
 However, the non-contracting condition was used in \cite{JaeTak20} only to ensure the non-existence of  points with negative pointwise Lyapunov exponent.
Although we do not assume the Bowen-Series map $f$ satisfies \eqref{NC}, 
the formulas in \cite[Main~Theorem(b)]{JaeTak20} remain intact for 
the level sets $\mathscr L(\alpha)$
since all points in these sets have non-negative pointwise Lyapunov exponents.
This proves the desired formula in (b).

Let $\alpha\in[\alpha_-,\alpha_+]\setminus \{0\}$.
In order to remove the limit $\varepsilon\to0$ in \eqref{L-dimension}, we use Lemma~\ref{meas-rep} to transfer the problem to $\mathcal M(X,\sigma)$. 
Since the function $\varphi:X\rightarrow \mathbb{R}$ 
  is continuous and the entropy is upper semicontinuous on the compact  space $\mathcal M(X,\sigma)$, 
 we can choose a convergent sequence $\{\mu_n\}_{n=1}^{\infty}$ 
 in $\mathcal M(X,\sigma)$ with positive entropy such that its weak* limit point 
 $\mu_{\infty}$ is an expanding measure satisfying 
 $\dim(\mu_{\infty}\circ\pi^{-1})=b(\alpha).$
This yields the desired formula in (c).
 \end{proof}



\subsection{Bowen's formula}\label{bow-sec}
The next type of formula, first established  in \cite{Bow79} for Fuchsian groups without parabolic elements, is referred to as Bowen's formula. It is known 
for conformal 
Graph Directed Markov Systems \cite[Theorem~4.2.13]{MauUrb03} which, in dimension one, correspond to uniformly expanding finitely irreducible Markov maps. 
Bowen's formula is also known 
for parabolic Iterated Function Systems \cite[Theorem~8.3.6]{MauUrb03}
and 
essentially free Kleinian groups with parabolic elements \cite{KesStr04}.

 \begin{prop}[Bowen's formula]\label{delta0}
We have 
\[\delta_G=\min\{\beta \ge 0 \colon P(\beta)=0\}.\]
\end{prop}
\begin{proof}
Put $\delta_0=\sup\{\dim(\mu)\colon\mu\in\mathcal M(\Lambda,f)\}$.
Since $f$ is saturated by Proposition~\ref{saturation}, we have
$\delta_0=\delta_G.$ Set
$\delta_1={\rm min}\{\beta\geq0\colon P(\beta)=0\}.$
It suffices to show 
$\delta_0=\delta_1$.
By definition, $\delta_0\geq \dim(\mu)$ holds for any
expanding measure $\mu \in \mathcal M(\Lambda,f)$.
Hence, $P(\delta_0)\leq0$ and so $\delta_0\geq\delta_1$.
Suppose for a contradiction that  $\delta_0>\delta_1$. Then there  exists $\varepsilon>0$ such that $\delta_0>\delta_1+\varepsilon$ and an expanding measure $\mu$ such that
$\dim(\mu)>\delta_1+\varepsilon$, and so $P(\delta_1+\varepsilon)>0$.
On the other hand,
 by the definition of $\delta_1$ and the monotonicity of pressure, we have  $P(\delta_1+\varepsilon) \leq P(\delta_1+\varepsilon/2)\leq0$, and a contradiction arises.
 Therefore $\delta_0=\delta_1$ holds.
\end{proof}

\subsection{Dimension formula for level sets in terms of pressure}\label{dim-sec}
We call  $\mu\in\mathcal M(\Lambda,f)$ 
satisfying $\dim(\mu)=\delta_G$ a {\it  measure of maximal dimension for $G$}. 
For the proof of the next proposition we refer the reader to Appendix~\ref{apdim}. 
\begin{prop}\label{nomax}
There exists a measure  of maximal dimension for $G$ if and only if $G$ has no parabolic element.
\end{prop}



\begin{lemma}\label{beta+zero}
If $G$ has a parabolic element, then $\lim_{\beta\nearrow\beta_+}P'(\beta)\ge 0.$
\end{lemma}
\begin{proof} 
 Proposition~\ref{press}(a) gives $\beta_+<\infty.$ 
Suppose for a contradiction that $\lim_{\beta\nearrow\beta_+}P'(\beta)<0.$
Take a sequence $\{\beta_n\}_{n=1}^{\infty}$ with $\beta_n\nearrow \beta_+$ 
as $n\to\infty$ and $\lim_{n\to\infty}P'(\beta_n)<0.$
Let $\mu_{\beta_n}$ be the equilibrium state for the potential $-\beta_n\log|f'|$  and 
let $\mu$ be an weak* accumulation 
point of $\{\mu_{\beta_n}\}_{n=1}^{\infty}$. Recall that
$X$ is a finite Markov shift where the entropy is upper semicontinuous,
and the function
$\varphi\colon X\to\mathbb R$ in \eqref{varphi} is continuous.
Hence,
$\mu$ is an equilibrium state for the potential $-\beta_+\log|f'|$, namely, $h(\mu)-\beta_+ \chi(\mu)=0$.
Since $P'(\beta_n)=-\chi(\mu_{\beta_n})$ by Proposition~\ref{press}(c), we have
$\chi(\mu)=\lim_{n\to\infty}\chi(\mu_{\beta_n})=-\lim_{n\to\infty}P'(\beta_n)>0.$ 
By Proposition~\ref{delta0}, $\mu$ is  a measure of maximal dimension for $G$, a contradiction to Proposition~\ref{nomax}.\end{proof}

\begin{prop}\label{press-II} If $G$ has a parabolic element, then the pressure is equal to zero on $[\beta_+,+\infty)$,
 $\beta_+=\delta_G$ and
the pressure function $P$ is $C^1$ on $\mathbb R$. Moreover, $P$ is strictly convex on $(-\infty, \delta_G)$. 
\end{prop}
\begin{proof}
 Lemma~\ref{alp-lem2} gives $\alpha_-=0$. By the definition of $\beta_+$, the pressure is equal to zero 
 on $[\beta_+,+\infty)$. 
 By Proposition~\ref{delta0}  we have  $\beta_+=\delta_G$.
 By Proposition~\ref{press}(b), $P$ is analytic on $(-\infty,\beta_+)$. The continuous differentiability of $P$ at $\beta=\delta_G$ is a consequence of  
 Lemma~\ref{beta+zero} and the convexity of $P$. This shows that $P$ is $C^1$ on $\mathbb R$.
Since $P$ is analytic, convex and  non-increasing on $(-\infty, \delta_G)$,  an elementary inductive argument on the  power series expansion of $P$ shows that
 either $P$ is affine on $(-\infty,\delta_G)$ or strictly convex on  $(-\infty,\delta_G)$.  The first case is ruled out by 
Lemma~\ref{beta+zero} and the assumption that $G$ is non-elementary which gives   
$\delta_G>0$.
\end{proof}
From Lemma~\ref{alp+-}, Propositions~\ref{p-equal}
 and Proposition~\ref{press-II},
we have
\[\alpha_+=-\lim_{\beta\searrow-\infty}P'(\beta)\ \text{and}\
\alpha_-=-\lim_{\beta\nearrow \beta_+}P'(\beta).\]
 By Proposition~\ref{press}(c) and Proposition~\ref{press-II}, and the implicit function theorem, there exists a strictly decreasing analytic function $\beta:(\alpha_-,\alpha_+)\rightarrow (-\infty,\beta_+) $ satisfying 
$-P'(\beta(\alpha))=\chi(\mu_{\beta(\alpha)})=\alpha$.
We have
\begin{equation}\label{limbeta}\lim_{\alpha\searrow\alpha_-}\beta(\alpha)=\beta_+\ \text{ and }\
\lim_{\alpha\nearrow\alpha_+}\beta(\alpha)=-\infty.\end{equation}

\begin{prop}\label{b-alpha}
 For all $\alpha\in (\alpha_-,\alpha_+)$ we have
\begin{equation} \label{eq-spectrum1}
\alpha b(\alpha)=P(\beta(\alpha))+ \alpha\beta(\alpha)\ \text{ and }\ b(\alpha)
=\frac{P^*(-\alpha)}{\alpha}.
\end{equation}
Moreover, the $\mathscr{H}$-spectrum 
is analytic on $(\alpha_-,\alpha_+)$ and satisfies
\begin{equation} \label{eq-spectrum2}
b'(\alpha)=\frac{-P(\beta(\alpha))}{\alpha^2}.
\end{equation}
\end{prop}

\begin{proof}
We have
$P(\beta(\alpha))+\alpha\beta(\alpha)=
h(\mu_{\beta(\alpha)})-\beta(\alpha)\alpha+\alpha\beta(\alpha)
=h(\mu_{\beta(\alpha)})\leq\alpha b(\alpha)$ where
the last inequality follows from Proposition~\ref{l-formula}(c).
Again by Proposition~\ref{l-formula}(c),  there exists an expanding measure
 $\mu\in\mathcal M(\Lambda,f)$ such that $\chi(\mu)=\alpha$ and $\dim(\mu)=b(\alpha)$.
Then $\alpha b(\alpha)=h(\mu)= h(\mu)-\beta(\alpha)\alpha+\alpha\beta(\alpha)
\leq P(\beta(\alpha))+\alpha\beta(\alpha)$, and so the first equality in \eqref{eq-spectrum1} holds.

The second equality in \eqref{eq-spectrum1} follows from the first. 
Since $P$ and $\beta$ are analytic,    the $\mathscr{H}$-spectrum is analytic on $(\alpha_-,\alpha_+)$ by the first formula in \eqref{eq-spectrum1}.
Differentiating the first equality in \eqref{eq-spectrum1}, and  combining with the first equality  in \eqref{eq-spectrum1} and the fact that $P'(\beta(\alpha))=-\alpha$,  yields the equality in \eqref{eq-spectrum2}.
\end{proof}

\subsection{Proof of the Main~Theorem}\label{proof-thma}
Lemma~\ref{alp-lem1} gives $\alpha_{-}<\alpha_+$.  By Proposition~\ref{l-formula}(a) we have $\mathscr{H}(\alpha)\neq\emptyset$
if and only if $\alpha\in[\alpha_-,\alpha_+]$. The analyticity of the $\mathscr{H}$-spectrum on $(\alpha_-,\alpha_+)$ is due to
 Proposition~\ref{b-alpha}.
 The formula \eqref{L-dimension} implies that  the $\mathscr{H}$-spectrum is 
upper semicontinuous on $[\alpha_-,\alpha_+]$.
The lower semicontinuity of the spectrum at  
 $\alpha\in\{\alpha_-,\alpha_+\}\setminus \{0\}$
can be derived from Proposition~\ref{l-formula}(c). To prove this, we may assume that  $b(\alpha)>0$ and denote by $\mu \in \mathcal{M}(\Lambda,f)$  an expanding measure  with $\dim(\mu)=b(\alpha)$. Let $\alpha' \in (\alpha_-,\alpha_+)$ with  $\alpha'\neq \alpha$ and $\mu' \in \mathcal{M}(\Lambda,f)$ such that $\chi(\mu')=\alpha'$. Applying  Proposition \ref{l-formula}(c)  to convex combinations   $p\mu+(1-p)\mu'$ and letting $p\rightarrow 1$    shows that the spectrum is lower semicontinuous at $\alpha$. The remaining case $\alpha_-=0$ is  covered by  \cite[Main~Theorem(b)(ii)]{JaeTak20}. We have thus shown that $b$ is continuous on $[\alpha_-,\alpha_+]$. By the second formula in \eqref{eq-spectrum1} we have $b(\alpha)=P^*(-\alpha)/\alpha$ for $\alpha\in (\alpha_-,\alpha_+)$. Since $P^*$ is continuous on $[\alpha_-,\alpha_+]$, this formula extends to $[\alpha_-,\alpha_+]\setminus\{0\}$.

To complete the proof of (a), we define \begin{equation}\label{ag}\alpha_G=-P'(\delta_G).\end{equation} 
If $G$ has no parabolic element, we have
 $\beta(\alpha_G)=\delta_G$, and so $b'(\alpha_G)=0$, 
   and $b'(\alpha)(\alpha-\alpha_G)<0$ for $\alpha\in(\alpha_-,\alpha_+)\setminus\{\alpha_G\}$
 by \eqref{eq-spectrum2}.
If $G$ has a parabolic element, then
Proposition~\ref{press-II} gives $\alpha_G=0$, and \cite[Main~Theorem(b)(ii)]{JaeTak20} gives
$\lim_{\alpha\searrow\alpha_-}b(\alpha)=b(\alpha_-)=\delta_G.$
 Moreover,
 \eqref{eq-spectrum2} implies $b'(\alpha)<0$ for  $\alpha\in(\alpha_-,\alpha_+)$, and so
the $\mathscr{H}$-spectrum is strictly monotone decreasing on $[\alpha_-,\alpha_+]$.

 Finally, it follows from  \eqref{limbeta} and \eqref{eq-spectrum2} 
 that 
$\lim_{\alpha\nearrow\alpha_+}b'(\alpha)
=-\infty$. 
Similarly, if $G$ has no parabolic element, then  \eqref{limbeta} and \eqref{eq-spectrum2} 
 give 
$\lim_{\alpha\searrow\alpha_-}b'(\alpha)
=+\infty$. 
The proof of (a) is complete.
The assertions in (b) follow from Proposition~\ref{press-II}.
\qed
\medskip


 Finally we show that the regularity of the pressure is related to the existence of an inflection point in the spectrum. 
\begin{prop}\label{schematic-prop}
If $G$ has a parabolic element and the geometric pressure function is $C^2$, then 
$P''(\delta_G)=0$ and the $\mathscr{H}$-spectrum has an inflection point. 
\end{prop}
\begin{proof}
Recall that $\beta(\alpha)$ is the unique solution of the equation $P'(\beta)+\alpha=0$. By the implicit function theorem, $\alpha\in(\alpha_-,\alpha_+)\mapsto\beta(\alpha)$ is differentiable and \begin{equation}\label{beta-der}\beta'(\alpha)=-\frac{1}{P''(\beta(\alpha))}.\end{equation}
We apply l'H\^opital's rule to \eqref{eq-spectrum2} together with \eqref{beta-der} to obtain 
$\lim_{\alpha\searrow\alpha_-}b'(\alpha)=-\lim_{\beta\nearrow\delta_G}1/(2P''(\beta))$ when one of the two limits exists. Since
  $P''(\beta)>0$ for $\beta<\delta_G$, we obtain
$\lim_{\alpha\searrow\alpha_-}b'(\alpha)=-\infty$. Since
$\lim_{\alpha\nearrow\alpha_+}b'(\alpha)=-\infty$ by the Main~Theorem,
we must have $b''(\alpha_*)=0$ for some $\alpha_*\in(\alpha_-,\alpha_+)$.
\end{proof}

\appendix
\def\thesection{\Alph{section}}
\section{}

\subsection{Proof of Proposition~\ref{press}}
 All the statements for $G$ without parabolic elements are well known \cite{Bow75,Rue04}, since some power of $f$ is uniformly expanding \cite[Theorem~5.1]{Ser81b} in this case.
Hence, we assume $G$ has a parabolic element.
Our strategy is to apply to
 $\tilde X$
the results in \cite{MauUrb03} 
on the thermodynamic formalism 
for countable Markov shifts.

Let
$\tilde\sigma\colon
\tilde X\to\tilde X$ denote the left shift. 
We write $\tilde\pi$ for $\pi_{\tilde X}$ and
 $\tilde t$ for $t\circ\tilde\pi$.
Note that $\tilde t$ is constant on each $1$-cylinder
$\tilde\varDelta(\tilde a)$ in $\tilde X$.
Let $\tilde t(\tilde{a})$ denote
the constant value of $\tilde t$ on each partition element $\tilde\varDelta(\tilde a)$.
For $n\geq1$ let 
\[\tilde S(n)=\{\tilde a\in \tilde S\colon \tilde t(\tilde a)=n\}.\]
\begin{lemma}\label{cardinal} 
 For any $n\geq 2$ we have
$\#\tilde S(n)\leq (\#S)^3$.
\end{lemma}
\begin{proof}From the definition of the Markov map $\tilde f$ in Section~\ref{c-ind}, for each 
$\tilde a\in \tilde S$ with $\tilde t(\tilde a)=n\geq2$,
there exists a unique element $\omega_0\omega_1\cdots\omega_{n}$ of $E^{n+1}(X)$ with $\varDelta(\omega_0\omega_1\cdots\omega_{n})=\tilde\varDelta(\tilde a)$.
Let $v$ denote the cusp that is contained in ${\rm cl}(\varDelta(\omega_1))$. Since $f^i(v)\in{\rm cl}(\varDelta(\omega_{i+1}))$ for $0\leq i\leq n-1$,
 the sequence $\omega_0\omega_1\cdots\omega_{n}$
is determined by the three symbols $\omega_0$, $\omega_1$ and $\omega_n$ in $S$. 
hence the desired inequality follows.\end{proof}

The next lemma follows from \cite[Lemma~2.8]{BowSer79}.
\begin{lemma}\label{mild-del}
For any
$\tilde a\in \tilde S$ and $\xi\in\tilde\varDelta(\tilde a)$ we have
 $
|({\tilde f})'\xi| \asymp \tilde t(\tilde a)^2.
$
\end{lemma}

For $(\beta,\zeta)\in\mathbb R^2$ we define
an induced potential $\Phi_{\beta,\zeta}\colon \tilde X\to \mathbb R$ by
\[\Phi_{\beta,\zeta}(\tilde x)=-\beta\log|(\tilde f)'\tilde\pi(\tilde x)|-\zeta\tilde t(\tilde x),\]
and an induced pressure
\[\mathscr{P}(\beta,\zeta)=\lim_{n\to\infty}\frac{1}{n}\log\sum_{\tilde\omega_0\cdots\tilde\omega_{n-1}\in E^n(\tilde X)}\sup_{[\tilde\omega_0\cdots\tilde\omega_{n-1}]}\exp\left(\sum_{k=0}^{n-1}\Phi_{\beta,\zeta}\circ\tilde\sigma^k\right).\]
Since logarithm of the series  is sub-additive in $n$, this limit exists and is never $-\infty$.
By \cite[Theorem~2.1.8]{MauUrb03}, the variational principle holds:
\[\mathscr{P}(\beta,\zeta)=\sup\left\{h(\tilde\mu)+\int \Phi _{\beta,\zeta} d\tilde\mu\colon\tilde\mu\in\mathcal M(\tilde X,\tilde\sigma),\int\Phi _{\beta,\zeta}d\tilde\mu>-\infty\right\}.\]
In the case $\mathscr{P}(\beta,\zeta)<+\infty$,
measures which attain this supremum are called {\it  equilibrium states} for the potential $\Phi_{\beta,\zeta}$.

We aim to verify sufficient conditions in
\cite[Theorem~2.2.9, Corollary~2.7.5]{MauUrb03}\footnote{The term ``bounded function'' in \cite[Corollary~2.7.5]{MauUrb03}  should be read ``function bounded from above''.} for the existence and uniqueness of a shift-invariant Gibbs state and the equilibrium state for the potential $\Phi_{\beta,\zeta}$.
We say $\Phi_{\beta,\zeta}$ is {\it  summable} if
\[
\sum_{\tilde a\in \tilde S}\sup_{[\tilde a]}\exp\Phi_{\beta,\zeta}<+\infty.\]
It is easy to see that the summability of $\Phi_{\beta,\zeta}$
implies the finiteness of
$\mathscr{P}(\beta,\zeta)$.


\begin{lemma}\label{pr}
The potential $\Phi_{\beta,\zeta}$   is summable if $\zeta>0$. 
\end{lemma}

\begin{proof}
By combining Lemmas \ref{cardinal} and \ref{mild-del}.  
\end{proof}

Recall that $d_{\tilde X}$ denotes the metric on $\tilde X$.
A function $\Psi\colon\tilde X\to\mathbb R$
is {\it  locally H\"older continuous} if there exist
constants  $C>0$ and $\theta\in(0,1]$ such that for any $\tilde a\in\tilde S$ and all $\tilde x,\tilde y\in[\tilde a]$ 
 we have
\[|\Psi(\tilde x)-\Psi(\tilde y)|\leq
C(d_{\tilde X}(\tilde x,\tilde y))^\theta.\]

\begin{lemma}\label{l-holder'''}
For any $(\beta,\zeta)\in\mathbb R^2$,
$\Phi_{\beta,\zeta}$ is locally H\"older continuous.
\end{lemma}
\begin{proof}
Let $\tilde a\in\tilde S$ and let $\tilde x,\tilde y\in[\tilde a]$.
Let $n\geq1$ be such that
$d_{\tilde X}(\tilde x,\tilde y)=e^{-n}$.
By the mean value theorem and 
Proposition~\ref{z-series-induce} there exists $\alpha_0>0$ such that $|\tilde\pi(\tilde\sigma\tilde x) -\tilde\pi(\tilde\sigma\tilde y)| \ll  e^{-\alpha_0 (n-1)}$. Combining this with the R\'enyi condition in  \cite[Lemma~2.8]{BowSer79} we obtain
\[|\log|(\tilde f)'\tilde\pi(\tilde x)|-\log|(\tilde f)'\tilde\pi(\tilde y)|| \ll  e^{-\alpha_0 (n-1)}.
\]
This implies that
$\log|(\tilde f)'|\circ\tilde\pi$ is locally H\"older continuous with 
$\theta=\min\{\alpha_0/2,1\}$. Moreover, $\tilde t$ is
 locally H\"older continuous
since it is constant on each induced $1$-cylinder.
Hence, $\Phi_{\beta,\zeta}$ is locally H\"older continuous. \end{proof}

\begin{lemma}\label{irreducible}
 The Markov map $\tilde{f}$ with the Markov partition 
$(\tilde\varDelta(\tilde a))_{\tilde a\in \tilde S}$ is finitely irreducible.  
\end{lemma}
\begin{proof}By the definition of $\tilde f$  and  the transitivity of the Markov  map $f$, the proof is straightforward. \end{proof}

By \cite[Corollary~2.7.5]{MauUrb03} together with Lemmas~\ref{pr}
and \ref{l-holder'''},
for any $\beta\in(-\infty,\beta_+)$
there exists a unique $\tilde \sigma$-invariant Gibbs state  $\tilde\mu_\beta$ for $\Phi_{\beta, P(\beta)}$, namely, there exists a constant $C\geq1$ such that 
for any $\tilde x\in\tilde X$ and $n\geq1$,
\begin{equation}\label{gibbs}
C^{-1}\leq\frac{\tilde\mu_\beta[\tilde x_0\cdots\tilde x_{n-1}]}{
\exp\left(-\mathscr{P}(\beta,P(\beta))n+\sum_{k=0}^{n-1}\Phi_{\beta,P(\beta)}
(\tilde\sigma^k\tilde x)\right)}\leq C.\end{equation}


\begin{lemma}\label{finite-int}
If $\beta\in(-\infty,\beta_+)$ then
$\int \tilde t d\tilde\mu_{\beta}<+\infty$ and
$\int \Phi_{\beta,P(\beta)}d\tilde\mu_\beta>-\infty.$
\end{lemma}
\begin{proof}
Let  $\beta < \beta_+$. Let $C$ denote the constant given by \eqref{gibbs}.  Lemma~\ref{mild-del} yields
\[
\sum_{n=1}^{\infty} n\sum_{\tilde a\in \tilde S(n)} \tilde\mu_\beta[\tilde a]\ll
 Ce^{-\mathscr{P}(\beta,P(\beta))}\sum_{n=1}^{\infty}  
 n^{1+2\beta} e^{-P(\beta)n},
\]
which is finite since $P(\beta)>0$. Hence,  by Lemma~\ref{cardinal},  
 $\int \tilde t d\tilde\mu_{\beta}=\sum_{\tilde a\in \tilde S} \tilde t(\tilde a)\tilde\mu_\beta[\tilde a]<+\infty$, and also
 $\int\log|({\tilde f})'|\circ\tilde\pi d\tilde\mu_{\beta}<+\infty$. Therefore
 $\int \Phi_{\beta,P(\beta)}d\tilde\mu_\beta>-\infty.$
 \end{proof}
  By \cite[Theorem~2.2.9]{MauUrb03}
together with Lemma~\ref{finite-int}, $\tilde\mu_\beta$ is the unique equilibrium state for the potential 
$\Phi_{\beta,P(\beta)}$, namely
\begin{equation}\label{mathp}
\mathscr{P}(\beta,P(\beta))=
h(\tilde \mu_\beta)+\int
-\beta\log|({\tilde f})'|\circ\tilde\pi-P(\beta)\tilde t d\tilde\mu_\beta.
\end{equation}
The measure given by
\[\mu_\beta=\frac{1}{\int \tilde t d\tilde\mu_\beta}\sum_{n=0}^{\infty}
\tilde\mu_\beta|_{\{\tilde t>n\}}\circ (f^n\circ\tilde\pi)^{-1}\]
belongs to $\mathcal M(\Lambda,f)$ and
by Abramov-Kac's formula \cite[Theorem~2.3]{PesSen08} 
satisfies
\begin{equation}\label{induce-p-eq1}
\mathscr{P}(\beta,P(\beta))=
\left(h(\mu_\beta)-\beta\chi(\mu_\beta)-P(\beta)\right)\int \tilde t d\tilde\mu_\beta.
\end{equation}


\begin{lemma}\label{induce-p-lem}
For all $\beta\in(-\infty,\beta_+)$,
$\mathscr P(\beta,P(\beta))=0$.
\end{lemma}
\begin{proof}
Let $\varepsilon>0$.
 From Lemma~\ref{meas-rep} and the fact that any measure in $\mathcal M(X,\sigma)$ is approximated in the weak* topology by ergodic ones with similar entropy
  \cite[Theorem~B]{eizen94}, it follows that there
 exists an ergodic $\nu\in\mathcal M(\Lambda,f)$ with $h(\nu)>0$ and
 $h(\nu)-\beta\chi(\nu)> P(\beta)-\varepsilon$.
Since $\Lambda\setminus\bigcup_{n=0}^\infty f^{-n}(V_c)\subset \bigcup_{n=0}^\infty f^{-n}(\tilde\Lambda)$, measures in $\mathcal M(\Lambda,f)$  supported in the complement of $\tilde\Lambda$
have zero entropy, and so
 $\nu(\tilde\Lambda)>0$.
The normalized restriction $\bar\nu$ of $\nu$ to $\tilde\Lambda$ is $\tilde {f}$-invariant and so the measure 
$\tilde\nu=\bar\nu\circ\tilde\pi^{-1}$ belongs to $\mathcal M(\tilde X,\tilde\sigma)$.
The variational principle for the potential 
$\Phi_{\beta,P(\beta)}$ 
yields
\[\begin{split}\mathscr{P}(\beta,P(\beta)-\varepsilon)&\geq  h(\tilde\nu)+\int-\beta\log|(\tilde f)'|\circ \tilde \pi-(P(\beta)-\varepsilon)\tilde t d\tilde\nu\\
&=
\left(h(\nu)-\beta\chi(\nu)-P(\beta) +\varepsilon\right) \int \tilde t d\tilde\nu\geq0.\end{split}\]
Since $\beta< \beta_+$, we have $P(\beta)>0$.  By Lemma~\ref{pr}, $\mathscr P(\beta,P(\beta)-\varepsilon)$ is finite for sufficiently small $\varepsilon\geq0$. The variational principle \cite[Theorem~2.1.8]{MauUrb03} implies that the non-negative function $\varepsilon\mapsto \mathscr P(\beta,P(\beta)-\varepsilon)$ is convex on a neighborhood of $\varepsilon=0$. Hence
we obtain $\mathscr P(\beta,P(\beta))\geq0$.
Combining this with \eqref{induce-p-eq1} we conclude that
$\mathscr P(\beta,P(\beta))=0$.
\end{proof}
From \eqref{induce-p-eq1} and Lemma~\ref{induce-p-lem},
 $\mu_\beta$ is an equilibrium state of $f$ for the potential $\beta\phi$.
 From the uniqueness of $\tilde\mu_\beta$, such an equilibrium state is unique, namely
  $\mu_\beta$ is the unique equilibrium state of $f$ for this potential.
  Since $P(0)>0$, $\alpha_-=0$ and $P(1)=0$, we have $0<\beta_+\leq1$. This completes the proof of (a).

Next we show the analyticity of the pressure.
Let $\beta_0\in(-\infty,\beta_+)$.
By 
\cite[Theorem~2.6.12]{MauUrb03} together with
Lemma~\ref{pr}, $\mathscr{P}(\beta,\zeta)$ is analytic at $(\beta_0,P(\beta_0))$, and so can be extended to a holomorphic function in a complex 
neighborhood of $(\beta_0,P(\beta_0))$. Note that the analyticity results in \cite{MauUrb03} continue to hold for finitely irreducibile shift spaces, see also \cite{PU21}. 
Lemma~\ref{induce-p-lem} gives $\mathscr{P}(\beta_0,P(\beta_0))=0$, and \eqref{mathp}
 shows
$\frac{\partial\mathscr{P}}{\partial \zeta}(\beta,P(\beta))=-\int\tilde t d\tilde\mu_\beta\neq0$.
By the implicit function theorem for holomorphic functions, the pressure
is analytic at $\beta=\beta_0$.
 This completes the proof of (b).


Finally, we verify (c). By \cite[Theorem~2.6.13]{MauUrb03},
$\frac{\partial\mathscr{P}}{\partial \beta}(\beta,P(\beta))=-\int\log|\tilde f| d\tilde\mu_\beta$.
The implicit function theorem gives
\[P'(\beta)=-\frac{\frac{\partial\mathscr{P}}{\partial \beta}(\beta,P(\beta))}{\frac{\partial\mathscr{P}}{\partial \zeta}(\beta,P(\beta))}=-\frac{\int\log|\tilde f| d\tilde\mu_\beta}{\int\tilde t d\tilde\mu_\beta}=-\chi(\mu_\beta),\]
as required.
The proof of Proposition~\ref{press} is complete.\qed

\subsection{Proof of Proposition~\ref{nomax}}\label{apdim}
It is well known that $G$ has a measure of maximal dimension if $G$ has no parabolic element. Now assume that $G$ has a parabolic element and assume for a contradiction that there exists a measure of maximal dimension $\mu$. 
From Proposition~\ref{delta0}, 
$\mu\circ\pi$ is an equilibrium state for the potential $-\delta_G\log|f'|\circ \pi$. Since $\delta_G>0$ by Proposition~\ref{delta0}, we have $h(\mu\circ\pi)>0$. Hence, 
$\mu(\tilde\Lambda)>0$.
The normalized restriction 
$\tilde\mu$ of 
$\mu\circ\pi$ to 
$\pi^{-1}(\tilde\Lambda)$
is 
$\tilde\sigma$-invariant
and satisfies 
$\int \tilde t d\tilde\mu<+\infty$ by Kac's formula.  Moreover, 
\[
 h(\tilde \mu)+\int \Phi_{\delta_G,0} d\tilde \mu= \left( h(\mu)-\delta_G \chi(\mu) \right) \int \tilde t d\tilde\mu =0.
\]
Proposition~\ref{delta0} and   approximations by finite subsystems together imply
$\mathscr{P}(\delta_G,0)\leq0$.
Hence, $\mathscr{P}(\delta_G,0)=0$
and by \cite[Theorem~2.1.9]{MauUrb03}, $\Phi_{\delta_G,0}$ is summable.
By \cite[Corollary~2.7.5]{MauUrb03},
 $\tilde \mu$ is the unique Gibbs-equilibrium state for 
the potential $\Phi_{\delta_G,0}$. 
By Lemma~\ref{mild-del}, we obtain a constant $C>0$ with $\int \tilde t d\tilde\mu \ge C \sum_{n=1}^{\infty} n\cdot  n^{-2\delta_G}=+\infty$, a contradiction.\qed
\subsection{Typical homological growth rates}\label{typical-h}
 Recall that  $\alpha_G$ 
 is the unique maximal point of the $\mathscr{H}$-spectrum: 
 $b(\alpha_G)=\delta_G$ by
 The Main~Theorem.
 It is well known \cite[Theorem~1.2]{Bea71} that
 $G$ is of the second kind if and only if 
$\delta(G)<1$. Hence, $b(\alpha_G)=1$ if and only if $G$ is of the first kind.
 Even more, the following holds.
 \begin{prop}\label{intermittency}
 If $G$ is of the first kind, 
 then $|\Lambda\setminus \mathscr{H}(\alpha_G)|=0$.
 \end{prop}
 \begin{proof}
  By Proposition~\ref{coincide-lem}, it is enough to show that $(1/n)\log|(f^n)'|$ converges Lebesgue a.e. to the constant $\alpha_G$ in \eqref{ag}
 as $n\to\infty$. 
 If $G$ has no parabolic element,
 there exists an ergodic $f$-invariant probability measure
 $\mu_{\rm ac}$ that is absolutely continuous with respect to the Lebesgue measure.
 Since $\sigma\colon X\to X$ is transitive, the support of $\mu_{\rm ac}$ is equal to $\Lambda$. 
  By Birkhoff's ergodic theorem,
  $(1/n)\log|(f^n)'|$ converges Lebesgue a.e.
  to the Lyapunov exponent of $\mu_{\rm ac}$,
   namely, the Lebesgue measure of the set $\Lambda\setminus \mathscr{H}(\chi(\mu_{\rm ac}))$ is $0$ and
    $\chi(\mu_{\rm ac})=\alpha_G$ holds.
 
 If $G$ has a parabolic element, then $\alpha_G=0$ by the Main~Theorem.
  It suffices to show that for any open set $U$ containing all neutral periodic points of $f$,
 \begin{equation}\label{intermittency-eq}\lim_{n\to\infty}\frac{1}{n}\#\{0\leq i\leq n-1\colon f^i(\xi)\in U\}=1\end{equation}
 for Lebesgue almost every $\xi\in\Lambda$.
 The $f$-invariant measure $\mu_{\rm ac}=\sum_{n=0}^{\infty}
\tilde\mu_1|_{\{\tilde t>n\}}\circ (f^n\circ\tilde\pi)^{-1}$ is ergodic, infinite and  absolutely continuous with respect to the Lebesgue measure. 
Moreover,
the density of $\mu_{\rm ac}$ is positive everywhere and infinite only at the neutral periodic points.  
It is well known that for $h\in L^1(\mu_{\rm ac})$ we have $\lim_{n\rightarrow \infty} (1/n) \sum_{k=0}^{n-1} h\circ f^k=0$ $\mu_{\rm ac}$-a.e.
 Since  $\Lambda\setminus U$ has finite $\mu_{\rm ac}$ measure, taking $h$ as the indicator of $\Lambda\setminus U$ proves \eqref{intermittency-eq}  for Lebesgue almost every $\eta\in\Lambda$.
 \end{proof}

\subsection*{Acknowledgments}
We thank Hiroyasu Izeki for fruitful discussions.
JJ was supported by the JSPS KAKENHI 21K03269.
HT was supported by the JSPS KAKENHI 
19K21835 and 20H01811.

\end{document}